\documentclass[11pt]{amsart}
\usepackage[T1]{fontenc}
\usepackage{amssymb,amsmath, amsthm, amsfonts, tikz-cd}

\usepackage{graphicx}
\usepackage{listings}
\usepackage[margin=.75in]{geometry}
\usepackage{lstautogobble}
\usepackage{enumerate}
\usepackage[shortlabels]{enumitem}
\usepackage{thmtools}
\usepackage{thm-restate}
\usepackage{amsthm}
\usepackage{verbatim}

\usepackage{mathtools}
\usepackage{physics}
\usepackage{color}
\usepackage{hyperref}
\usepackage[capitalise]{cleveref}
\usepackage[final]{showlabels}
\usepackage[bottom]{footmisc}
\crefformat{equation}{(#2#1#3)}

\usepackage{float}
\restylefloat{table}

\usepackage{mathrsfs}
\setlist{  
  listparindent=\parindent,
  parsep=0pt,
}

\theoremstyle{plain}
\newtheorem{thm}{Theorem}[section]
\newtheorem{prop}[thm]{Proposition}
\newtheorem{lemma}[thm]{Lemma}
\newtheorem{cor}[thm]{Corollary}

\theoremstyle{definition}
\newtheorem{mydef}[thm]{Definition}

\newtheorem{remark}[thm]{Remark}

\numberwithin{equation}{section} 

\DeclarePairedDelimiter\ipp{\langle}{\rangle}

\DeclarePairedDelimiter{\paren}{\lparen}{\rparen}

\DeclarePairedDelimiter{\jp}{\langle}{\rangle}

\DeclareMathOperator{\supp}{supp}
\DeclareMathOperator{\diam}{diam}

\newcommand{\M}{{\mathcal{M}}}

\newcommand{\p}{{\partial}}

\renewcommand{\d}{\delta}

\newcommand{\R}{{\mathbb{R}}}
\newcommand{\C}{{\mathbb{C}}}
\newcommand{\N}{{\mathbb{N}}}

\newcommand{\Z}{{\mathbb{Z}}}

\renewcommand{\H}{{\mathcal{H}}}
\renewcommand{\P}{{\mathcal{P}}}
\newcommand{\T}{{\mathbb{T}}}
\newcommand{\E}{\mathbf{E}}
\newcommand{\g}{{\mathfrak{g}}}

\newcommand{\G}{{\mathfrak{G}}}

\newcommand{\Fr}{{\mathfrak{F}}}

\newcommand{\I}{{\mathcal{I}}}
\newcommand{\Sc}{{\mathcal{S}}}

\renewcommand{\M}{{\mathcal{M}}}

\newcommand{\tl}{\tilde}

\newcommand{\D}{\Delta}

\newcommand{\ph}{\phantom{=}}
\newcommand{\nn}{\nonumber}

\newcommand{\ol}{\overline}
\newcommand{\ul}{\underline}

\newcommand{\ux}{\underline{x}}

\newcommand{\usig}{\underline{\sigma}}

\newcommand{\ue}{\underline{\eta}}
\newcommand{\ur}{\underline{r}}

\newcommand{\ep}{\epsilon}
\newcommand{\vep}{\varepsilon}
\newcommand{\al}{\alpha}
\newcommand{\be}{\beta}
\newcommand{\ga}{\gamma}
\newcommand{\et}{\eta}

\newcommand{\om}{\omega}

\newcommand{\wh}{\widehat}
\newcommand{\ueta}{\underline{\eta}}

\renewcommand{\ueta}{\underline{\eta}}
\newcommand{\F}{{\mathcal{F}}}

\let\oldtocsection=\tocsection
 
\let\oldtocsubsection=\tocsubsection
 
\let\oldtocsubsubsection=\tocsubsubsection
 
\renewcommand{\tocsection}[2]{\hspace{0em}\oldtocsection{#1}{#2}}
\renewcommand{\tocsubsection}[2]{\hspace{1em}\oldtocsubsection{#1}{#2}}
\renewcommand{\tocsubsubsection}[2]{\hspace{2em}\oldtocsubsubsection{#1}{#2}}

\begin{document}

\title[Mean-Field Limit of Stochastic Point Vortices]{The Mean-Field Limit of Stochastic Point Vortex Systems with Multiplicative Noise}

\author[M. Rosenzweig]{Matthew Rosenzweig}
\address{  
Department of Mathematics\\ 
Massachusetts Institute of Technology\\
Headquarters Office\\
Simons Building (Building 2), Room 106\\
77 Massachusetts Ave\\
Cambridge, MA 02139-4307}
\email{mrosenzw@mit.edu}

\begin{abstract}
In \cite{FGP2011}, Flandoli, Gubinelli, and Priola proposed a stochastic variant of the classical point vortex system of Helmholtz \cite{Helmholtz1858} and Kirchoff \cite{Kirchoff1876} in which multiplicative noise of transport-type is added to the dynamics. An open problem in the years since is to show that in the mean-field scaling regime, in which the circulations are inversely proportional to the number of vortices, the empirical measure of the system converges to a solution of a two-dimensional Euler vorticity equation with multiplicative noise. By developing a stochastic extension of the modulated-energy method of Serfaty \cite{Serfaty2017, Serfaty2020} and Duerinckx \cite{Duerinckx2016} for mean-field limits of deterministic particle systems and by building on ideas introduced by the author \cite{Rosenzweig2020_PVMF, Rosenzweig2020_CouMF} for studying such limits at the scaling-critical regularity of the mean-field equation, we solve this problem under minimal assumptions. 
\end{abstract}

\maketitle

\section{Introduction}
\label{sec:intro}
\subsection{Point vortices}
\label{ssec:intro_PV}
Investigating the regularizing effect of multiplicative noise for ideal fluid flow, Flandoli, Gubinelli, and Priola \cite{FGP2011} proposed a stochastic variant of the classical point vortex model of Helmholtz \cite{Helmholtz1858} and Kirchoff \cite{Kirchoff1876} of the form
\begin{equation}
\label{eq:SPVM}
\begin{cases}
\displaystyle dx_{i,t} = \sum_{1\leq j\leq N; j\neq i}a_j\nabla^\perp\g(x_{i,t}-x_{j,t})dt + \sum_{k=1}^\infty \sigma_k(x_{i,t})\circ dW^k\\
\displaystyle x_{i,t}|_{t=0} = x_i^0
\end{cases}
\qquad i\in\{1,\ldots,N\},
\end{equation}
where $N$ is the number of vortices, $a_1,\ldots,a_N\in\R\setminus\{0\}$ are the vortex intensities, $x_1^0,\ldots,x_N^0\in\R^2$ are the pair-wise distinct initial positions, $\g(x)\coloneqq -\frac{1}{2\pi}\ln|x|$ is the Coulomb potential, and $\nabla^\perp=(-\p_{x_2},\p_{x_1})$. The stochastic perturbation of the deterministic dynamics is given by a random advection described by the smooth, divergence-free vector fields $\sigma_k$ and independent real Brownian motions $W^k$. The notation $\circ$ denotes the Stratonovich product. We refer to the stochastic differential equation (SDE) \eqref{eq:SPVM} as the \emph{stochastic point vortex model}. Remarkably, Flandoli et al. \cite{FGP2011} showed that the system \eqref{eq:SPVM} posed on the torus $\T^2$ has the property of full well-posedness, assuming the vector fields $\sigma_k$ satisfy a H\"ormander condition: for every choice of intensities and initial positions, there is a unique, global strong solution to \eqref{eq:SPVM}. This result for the stochastic dynamics is in stark contrast to the deterministic case ($\sigma_k\equiv 0$), for which it is easy to construct examples of initial configurations which merge together in finite time (e.g., see \cite[Section 4.2]{MP2012book}).

Solutions of the system \eqref{eq:SPVM} correspond to a special class of weak solutions to a two-dimensional (2D) \emph{stochastic incompressible Euler equation}, written in vorticy form as
\begin{equation}
\label{eq:SEul_S}
\begin{cases}
\p_t\xi + u\cdot\nabla\xi +\sum_{k=1}^\infty \sigma_k\cdot\nabla\xi\circ \dot{W}^k = 0\\
u=\nabla^\perp\g\ast\xi \\
\xi|_{t=0}=\xi^0
\end{cases} \qquad (t,x)\in [0,\infty)\times\R^2.
\end{equation}
Indeed, the empirical measure
\begin{equation}
\label{eq:EM}
\xi_N(t,x) \coloneqq \sum_{i=1}^N a_i\d_{x_i(t)}(x),
\end{equation}
satisfies the so-called \emph{weak vorticity formulation} of equation \eqref{eq:SEul_S} where the nonlinearity has been renormalized so as to remove the infinite self-interaction between the point vortices (see, for example, \cite{Schochet1995}). The equation \eqref{eq:SEul_S} corresponds to the well-known partial differential equation (PDE) for the evolution of the vorticity of the 2D incompressible Euler equation with an added stochastic transport term. Paralleling results for the deterministic equation, Brze\'{z}niak, Flandoli, and Maurelli \cite{BFM2016} have shown that weak solutions, and by implication classical solutions, to \eqref{eq:SEul_S} are globally well-posed for essentially bounded initial data in $L^1$ (cf. \cite{Wolibner1933, Yudovich1963}),\footnote{Strictly speaking, Brze\'{z}niak et al. consider periodic solutions (i.e. $\R^2$ is replaced by $\T^2$); however, their work can be adapted to treat case of Euclidean space considered in this article.} and Brze\'{z}niak and Maurelli \cite{BM2019} have shown that weak solutions still exist even starting from certain measures in $H^{-1}$ (cf. \cite{Delort1991, Majda1993, Schochet1995, Poupaud2002}), however it is unknown whether such measure-valued solutions are unique. We review the notion of solutions to this SPDE in \cref{ssec:pre_Eul}.

An open problem in the area of stochastic particle systems \cite{JW2017_survey} is to rigorously establish the SPDE \eqref{eq:SEul_S} as an \emph{effective description} of the dynamics of the SPVM \eqref{eq:SPVM} when the number $N$ of vortices is very large. To make such a connection precise, we consider the \emph{mean-field} scaling regime where the magnitudes of the vortex intensities are inversely proportional to the number of vortices (i.e. $|a_i|=1/N$), so that the total energy per vortex pair is finite as $N\rightarrow\infty$ and that up to a random error, the velocity field experienced by a single vortex is proportional to the average of the fields generated by the remaining point vortices. We limit our attention to the repulsive setting in which the intensities are identically signed--without loss of generality, each $a_i=1/N$. Formally, we expect that
\begin{equation}
\label{eq:intro_mfc}
\xi_N \xrightarrow[N\rightarrow\infty]{} \xi,
\end{equation}
in a suitable topology, where $\xi$ is a solution to the stochastic Euler equation \eqref{eq:SEul_S}. We refer to \eqref{eq:intro_mfc} as \emph{mean-field convergence}. Note that this persistence of stochasticity as $N\rightarrow\infty$ for multiplicative noise is in contrast to the point vortex model with \emph{additive noise}
\begin{equation}
\label{eq:SPVM_AN}
dx_{i,t} = \frac{1}{N}\sum_{{1\leq j\leq N}\atop{j\neq i}} \nabla^\perp\g(x_{i,t}-x_{j,t})dt + \sqrt{2\nu}d\tl{W}_t^i \qquad i\in\{1,\ldots,N\},
\end{equation}
where $\nu>0$ and $\{\tilde{W}^i\}_{i=1}^\infty$ are independent Brownian motions in $\R^2$. Indeed, the empirical measure for \eqref{eq:SPVM_AN} has been rigorously shown to converge in a suitable sense to a solution $\xi$ of the \emph{deterministic 2D Navier-Stokes equation} \cite{Osada1985, Osada1987, FHM2014, JW2018} (see also the recent works \cite{BJW2019, BJW2020} for further extensions)
\begin{equation}
\p_t\xi+u\cdot\nabla\xi = \nu\Delta\xi.
\end{equation}
 
There has been extensive work on studying the mean-field dynamics of the deterministic point vortex system (i.e. $\sigma_k\equiv 0$) beginning with work of Schochet \cite{Schochet1996} and continuing in more recent years with work of Jabin and Wang \cite{JW2018}, Serfaty \cite{Serfaty2020}, and the author \cite{Rosenzweig2020_PVMF}. We also mention that there has been extensive work on the related problem of convergence for and stability of vortex approximation methods for the deterministic Euler and related equations, e.g. see \cite{Rosenhead1931, Westwater1935, BM1982, GHL1990, GH1991, LX1995, LX2001, LZ2000} and \cite[Chapter 6]{MB2002}. But there has been little work studying the mean-field problem for the stochastic system \eqref{eq:SPVM}. To the best of our knowledge, the only result is by Coghi and Maurelli \cite{CM2020}, which shows that mean-field convergence holds in the periodic case if the Biot-Savart kernel is truncated to length scales much larger than the typical inter-vortex distance $N^{-1/2}$. Thus, showing mean-field convergence for the system \eqref{eq:SPVM}, and more generally systems of the form \eqref{eq:SPVM} with possibly more singular Biot-Savart kernels, is an open problem. 

\subsection{Main results}
\label{ssec:intro_MR}
We resolve the problem of rigorously establishing the stochastic Euler equation \eqref{eq:SEul_S} as the mean-field limit of the stochastic point vortex system \eqref{eq:SPVM}. As far as we are aware, this is the first such result for stochastic particle systems with singular interactions and multiplicative noise. Moreover, our result establishes the missing ``rigorous link,'' to borrow a phrase from Flandoli et al. \cite{FGP2011}, between the SPDE \eqref{eq:SEul_S} and the system of SDEs \eqref{eq:SPVM}, remaining to be established following Flandoli et al.'s work.

To state the theorem, we first clarify some assumptions for the underlying dynamics. We assume that the vector fields $\{\sigma_k\}_{k=1}^\infty$ are smooth, divergence-free, and satisfy the bounds
\begin{equation}
\begin{split}
\|\usig\|_{\ell_k^2 L_x^\infty(\N\times\R^2)}^2 &\coloneqq \sum_{k=1}^\infty \|\sigma_k\|_{L^\infty(\R^2)}^2 <\infty,\\
\|\nabla\usig\|_{\ell_k^2 L_x^\infty(\N\times\R^2)}^2 &\coloneqq \sum_{k=1}^\infty \|\nabla\sigma_k\|_{L^\infty(\R^2)}^2 < \infty.
\end{split}
\end{equation}
Furthermore, we assume that $\{W^k\}_{k=1}^\infty$ are independent real Brownian motions defined on and adapted to a fixed filtered probability space $(\Omega,\F, (\F_t)_{t\geq 0}, \mathbf{P})$ satisfying all the usual assumptions. We define the \emph{modulated energy}
\begin{equation}
\label{eq:intro_ME}
\Fr_N^{avg}(\ux_N(t),\xi(t)) \coloneqq \int_{(\R^2)^2\setminus\D_2}\g(x-y)d(\xi_N-\xi)(t,x)d(\xi_N-\xi)(t,y),
\end{equation}
where $\xi_N$ is the empirical measure from \eqref{eq:EM}, $\D_2$ is the diagonal of $(\R^2)^2$, and $\ux_N=(x_1,\ldots,x_N)$. We abuse notation by using the same symbol to denote both the measure and its density (when such a density exists).

\begin{thm}[Main result]
\label{thm:main}
There exists a constant $C>0$ such that the following holds. Let $\xi\in L^\infty(\Omega\times [0,T]; \P(\R^2)\cap L^\infty(\R^2))$ be a weak solution to the stochastic Euler equation \eqref{eq:SEul_S} with initial datum $\xi^0$ such that
\begin{equation}
\label{eq:main_ass}
\int_{\R^2}\ln\jp{x}d\xi^0(x) + \int_{(\R^2)^2}\ln\jp{x-y} d(\xi^0)^{\otimes 2}(x,y) < \infty.
\end{equation}
For $N\in\N$, let $\ux_N=(x_1,\ldots,x_N)$ be a strong solution to the stochastic point vortex model \eqref{eq:SPVM} with initial datum $\ux_N^0$. If for given $t\in [0,T]$, $N\in\N$ is sufficiently large so that $(\ln N)/N\leq \min\{e^{-1},\|\xi^0\|_{L^\infty(\R^2)}^{-1}\}$ and
\begin{equation}
\label{eq:intro_N_con}
\begin{split}
&C\paren*{\|\xi^0\|_{L^\infty(\R^2)} + \|\nabla\usig\|_{\ell_k^2 L_x^\infty(\N\times\R^2)}^2}t \\
&< \ln\ln\paren*{|\Fr_N^{avg}(\ux_N^0,\xi^0)| + \frac{C t (\|\xi^0\|_{L^\infty(\R^2)}+\|\nabla\usig\|_{\ell_k^2 L_x^\infty(\N\times\R^2)}^2)(\ln N)^2}{N}}^{-1},
\end{split}
\end{equation}
then $\Fr_N^{avg}(\ux_N(t),\xi(t))$ satisfies the inequality
\begin{equation}
\label{eq:main}
\begin{split}
&\E\paren*{|\Fr_N^{avg}(\ux_N(t),\xi(t))|} \\
&\leq \paren*{|\Fr_N^{avg}(\ux_N^0,\xi^0)| + \frac{C t (\|\xi^0\|_{L^\infty(\R^2)}+\|\nabla\usig\|_{\ell_k^2 L_x^\infty(\N\times\R^2)}^2)(\ln N)^2}{N}}^{e^{-Ct(\|\xi^0\|_{L^\infty(\R^2)} + \|\nabla\usig\|_{\ell_k^2 L_x^\infty(\N\times\R^2)}^2)}}.
\end{split}
\end{equation}
\end{thm}

The modulated energy $\Fr_N^{avg}(\ux_N(t),\xi(t))$ controls the Sobolev norm $\|\xi_N-\xi\|_{H^{s}(\R^2)}$, for any $s<-1$, up to an additive error which vanishes as $N\rightarrow\infty$. From this Sobolev control, one also obtains that the modulated energy controls weak-* convergence up to an additive error. These results are shown in \cref{prop:bes_conv}. Therefore, the meat of \cref{thm:main} is the effective bound \eqref{eq:main} for the absolute mean of $\Fr_N^{avg}(\ux_N(t),\xi(t))$. In fact, an examination of the proof of \cref{thm:main} reveals that one can control moments of the modulated energy of arbitrarily high order.

As observed in \cite[Remark 1.2(c)]{Duerinckx2016}, weak-* convergence of the initial empirical measures $\xi_N^0$ to $\xi^0$ and convergence of the N-body energy to the Coulomb energy of $\xi^0$ imply that $\Fr_N^{avg}(\ux_N^0,\xi^0)$ tends to zero, as $N\rightarrow\infty$. Moreover, by randomizing the initial data $\ux_N^0$, so that for each $N\in\N$, $x_{1,N}^0,\ldots,x_{N,N}^0$ are i.i.d. $\R^2$-valued random variables with density $\xi^0$, one can show $\Fr_N^{avg}(\ux_N^0,\xi^0)$ tends to zero almost surely (a.s.). In fact, by combining this randomization with control on higher moments of $\Fr_N^{avg}(\ux_N(t),\xi(t))$, one can show from the Kolmogorov three series theorem that for every $t$ fixed, $\Fr_N^{avg}(\ux_N(t),\xi(t))$ converges to zero as $N\rightarrow\infty$, a.s.

Observe that if that the $\sigma_k$ are identically zero so that the dynamics are deterministic, then $L^\infty(\R^2)$ is a function space which is \emph{invariant} under the scaling
\begin{equation}
\xi(t,x) \mapsto \xi(t,\lambda x), \qquad \lambda>0
\end{equation}
preserving the solution class to equation \eqref{eq:SEul_S}. Thus, not only is $L^\infty(\R^2)$ a critical function space for the (global) well-posedness of the equation by the results of Yudovich \cite{Yudovich1963} in the deterministic case and Brze\'{z}niak et al. \cite{BFM2016} in the stochastic case, the $L^\infty(\R^2)$ norm is also almost surely a conserved quantity (see \cref{rem:con_LP}).

\begin{remark}
The qualitative assumption \eqref{eq:main_ass} for $\xi^0$ is to ensure that both the stream function $\g\ast\xi$ and the energy $\ipp{\xi, \g\ast\xi}$ of the solution $\xi$ with initial datum $\xi^0$ are almost surely both well-defined, as the reader may check. For example, the logarithmic growth bound holds if $\xi^0$ has compact support and is bounded. Furthermore, \eqref{eq:main_ass} ensures that the key quantity $\Fr_N^{avg}(\ux_N,\xi)$ is well-defined.
\end{remark}

\begin{remark}
In the statement of \cref{thm:main}, we have taken the initial data to be deterministic. However, the proof of \cref{thm:main} may be adapted to assume that $\xi^0$ is a random measure with finite average $L^1\cap L^\infty$ norm and Coulomb energy, which almost surely satisfies the growth condition in \eqref{eq:main_ass}, and the $\ux_N^0$ are randomly distributed according to a law such that average energy and moment of inertia are finite.
\end{remark}

\begin{remark}
\label{rem:per}
We have chosen to work on $\R^2$, and not on $\T^2$ as in the prior works \cite{FGP2011, BFM2016, CM2020} on the SPVM and stochastic Euler equation, in order to parallel our prior work \cite{Rosenzweig2020_PVMF} on the deterministic mean-field problem and because the interaction potential $\g$ is explicit on $\R^2$. We expect that one can adapt our proof to the periodic setting, in particular using transference results (e.g., see \cite[Section 4.3]{grafakos2014c}) to carry over the singular integral estimates discussed in \cref{sec:app}.
\end{remark}

\subsection{Road map of proofs}
\label{ssec:intro_RM}
Our proof of \cref{thm:main} builds on the \emph{modulated-energy method} as developed for deterministic analogues of the system \eqref{eq:SPVM} by Duerinckx, Serfaty, and the author in the aforementioned works \cite{Duerinckx2016, Serfaty2020, Rosenzweig2020_PVMF, Rosenzweig2020_CouMF}.\footnote{While the specific form of the modulated-energy method employed in this article is inspired by the aforementioned works of Serfaty \cite{Serfaty2020} and Duerinckx \cite{Duerinckx2016}, we mention that this method is prefigured in earlier works, including by Brenier \cite{Brenier2000} and Lin and Zhang \cite{LZ2006_semi}. We also mention that this method is similar to the relative-entropy method, which has been used in recent years by Jabin and Wang \cite{JW2016, JW2018} for mean-field limits of interacting particle systems, as well as by others in numerous physical contexts (e.g., see \cite{SR2009}).} This method exploits a weak-strong stability principle for noise-less equations of the form \eqref{eq:SEul_S}. It is quantitative and avoids a need for control of the microscopic dynamics in terms of particle trajectories. But to date its use has been largely limited to deterministic mean-field problems. Only very recently has this method been extended by Bresch, Jabin, and Wang \cite{BJW2019, BJW2020} to treat stochastic particle systems with additive noise of the form \eqref{eq:SPVM_AN}. However, we are unaware of any applications of the modulated-energy method to systems with multiplicative noise.

As in \cite{Duerinckx2016, Serfaty2020, Rosenzweig2020_PVMF, Rosenzweig2020_CouMF}, the idea is to take as modulated energy the quantity $\Fr_N^{avg}(\ux_N(t),\xi(t))$ from the statement of \cref{thm:main}. One may consider this quantity a renormalization of the $\dot{H}^{-1}(\R^2)$ semi-norm, so as to remove the infinite self-interaction between point masses. The modulated energy $\Fr_N^{avg}(\ux_N(t),\xi(t))$ in our setting is now a continuous, \emph{stochastic} process which satisfies the SDE
\begin{equation}
\label{eq:ME_deriv}
\begin{split}
d\Fr_N^{avg}(\ux_N,\xi) &= \int_{(\R^2)^2 \setminus\D_2} \nabla\g(x-y)\cdot\paren*{u(x)-u(y)}d(\xi_N-\xi)(x)d(\xi_N-\xi)(y)dt \\
&\ph + \frac{1}{2}\sum_{k=1}^\infty \int_{(\R^2)^2\setminus\D_2} \nabla\g(x-y) \cdot \paren*{(\sigma_k\cdot\nabla)\sigma_k(x)-(\sigma_k\cdot\nabla)\sigma_k(y)}d(\xi_N-\xi)(x)d(\xi_N-\xi)(y)dt \\
&\ph + \sum_{k=1}^\infty\int_{(\R^2)^2\setminus \D_2} \nabla\g(x-y)\cdot \paren*{\sigma_k(x)-\sigma_k(y)}d(\xi_N-\xi)(x)d(\xi_N-\xi)(y)dW^k \\
&\ph + \frac{1}{2}\sum_{k=1}^\infty \int_{(\R^2)^2\setminus\D_2} \nabla^2\g(x-y) : \paren*{\sigma_k(x)-\sigma_k(y)}^{\otimes 2} d(\xi_N-\xi)(x)d(\xi_N-\xi)(y)dt.
\end{split}
\end{equation}
Here, the third term in the right-hand side is to be interpreted in the It\^{o} sense and the $:$ in the fourth term denotes the Frobenius inner product for $2\times 2$ matrices. The challenge is to show that the moments of the modulated energy must decay as $N\rightarrow 0$.

The third term in \eqref{eq:ME_deriv} has zero expectation and therefore can be ignored. In the deterministic case \cite{Rosenzweig2020_PVMF}, we showed how to estimate the first two terms assuming only that $u$ and $\sum_{k=1}^\infty (\sigma_k\cdot\nabla)\sigma_k$ are log-Lipschitz, assumptions easily seen to hold. We have reproduced the key estimate from \cite{Rosenzweig2020_PVMF} in \cref{sec:kprop} in the form of \cref{prop:FOLLkprop}. This estimate carries over to our present setting, as we can apply it point-wise in almost every realization of the noise. Thus, the new difficulty is to deal with the last term in \eqref{eq:ME_deriv}.

Unfortunately, \cref{prop:FOLLkprop} does not seem applicable to this last term, which is the second-order correction obtained when one formally converts the Stratonovich equation \eqref{eq:SEul_S} into an It\'{o} equation \eqref{eq:SEul_I} and which stems from the nonzero quadratic variation of Brownian motion. Our new insight, which ultimately enables us to accommodate the multiplicative noise, is that the last term in \eqref{eq:ME_deriv} can be treated as a perturbation of a special singular integral operator (SIO) known as a \emph{Calder\'{o}n $d$-commutator}, following terminology introduced by Christ and Journ\'{e} \cite{CJ1987}. We give an abbreviated review of these SIOs in \cref{ssec:app_CdC} and more generally, \cref{sec:app}. To illustrate this observation, we suppose for a moment that $\mu$ is a test function with zero mean. Then integrating by parts twice, we have that
\begin{equation}
\label{eq:intro_Tsig}
\begin{split}
&\int_{(\R^2)^2} \nabla^2 \g(x-y) : (\sigma(x)-\sigma(y))^{\otimes 2} d\mu(x)d\mu(y) \\
&=\int_{\R^2} (\nabla\g\ast\mu)(x)\cdot \nabla T_\sigma \nabla (\nabla \g\ast\mu)(x)dx,
\end{split}
\end{equation}
where
\begin{equation}
(T_{\sigma} \mu)(x) \coloneqq \int_{\R^2} \nabla^2\g(x-y) : (\sigma(x)-\sigma(y))^{\otimes 2} d\mu(y).
\end{equation}
By appealing to sophisticated SIO estimates of Christ and Journ\'{e} \cite{CJ1987} (see \cref{sec:app}), we can show that $\nabla T_\sigma \nabla$ is a bounded operator from $L^2(\R^2)$ to $L^2(\R^2; (\R^2)^{\otimes 2})$. It then follows from Cauchy-Schwarz that
\begin{equation}
\left|\int_{\R^2} (\nabla\g\ast\mu)(x)\cdot \nabla T_\sigma \nabla (\nabla \g\ast\mu)(x)dx\right| \lesssim \|\nabla T_{\sigma}\nabla\|_{2,2} \|\nabla\g\ast\mu\|_{L^2(\R^2)}^2,
\end{equation}
where $\|\nabla T_{\sigma}\nabla\|_{2,2}$ denotes the operator norm of $\nabla T_{\sigma}\nabla$. While $\nabla \g\ast (\xi_N-\xi)$ barely fails to be in $L^2$, we can smear out each point mass $\d_{x_i}$ to a small scale $\eta_i>0$ (see \eqref{eq:g_conv_smear} for the exact procedure) to obtain an approximate empirical measure  $\xi_N^{(\ueta_N)}$, where $\ueta_N \coloneqq (\et_1,\ldots,\et_N)$, satisfying
\begin{equation}
\|\nabla \g\ast (\xi_N^{(\ueta_N)}-\xi)\|_{L^2(\R^2)} \lesssim \Fr_N^{avg}(\ux_N,\xi) + o(1)
\end{equation}
as $N\rightarrow \infty$ (see \cref{cor:grad_H}). The error introduced by this approximation can also be shown to be $o(1)$ in the limit as $N\rightarrow\infty$.

Note that the preceding analysis so far has been path-wise. Ultimately by taking expectations, which eliminates the It\^o integral in the right-hand side of \eqref{eq:ME_deriv}, we obtain an integral inequality for the absolute mean of the modulated energy of the form
\begin{equation}
\begin{split}
\mathfrak{A}_N(t) &\leq |\Fr_N^{avg}(\ux_N^0,\xi^0)| +  C(\|\xi^0\|_{L^\infty}, \|\nabla\usig\|_{\ell_k^2 L_x^\infty})t\frac{(\ln N)^2}{N}  \\
&\ph + C(\|\xi^0\|_{L^\infty}, \|\nabla\usig\|_{\ell_k^2 L_x^\infty}) \int_0^t |\ln \mathfrak{A}_N(s)| \mathfrak{A}_N(s)ds
\end{split}
\end{equation}
where $C(\cdot,\cdot)$ is a constant depending on its two arguments and
\begin{equation}
\mathfrak{A}_N(t) \coloneqq \sup_{0\leq s\leq t} \E\paren*{|\Fr_N^{avg}(\ux_N(s),\xi(s))|}.
\end{equation}
An application of the Osgood lemma (see \cref{lem:Os} in \cref{ssec:pre_Os}), which is a generalization of the classical Gronwall-Bellman inequality, finally leads us to the estimate \eqref{eq:main} stated in \cref{thm:main}.

\subsection{Organization of article}
\label{ssec:intro_org}
Having presented the main results of this article and discussed their proofs, we briefly comment on the organization of the body of the article.

\cref{sec:pre} consists of notation and preliminary facts from harmonic analysis, concerning function spaces, singular integral estimates, and the Osgood lemma. Additionally, we review in \cref{ssec:pre_Eul} the well-posedness of the stochastic Euler equation \eqref{eq:SEul_S} as developed in \cite{BFM2016}. \cref{sec:CE} is devoted to properties of the Coulomb potential $\g$ and the modulated energy functional $\Fr_N^{avg}(\cdot,\cdot)$. Many of the results in this section are already contained in the articles \cite{Duerinckx2016, Serfaty2020, Rosenzweig2020_PVMF}. Therefore, we generally include only statements of the results and skip repeating the proofs.

\cref{sec:kprop} is the meat of this article, containing the proofs of the crucial \cref{prop:FOLkprop}, \cref{prop:FOLLkprop}, and \cref{prop:SOkprop}. We begin this section with an overview. The remaining three subsections correspond to the proofs of the three aforementioned propositions. Finally, \cref{sec:MR} is where we prove our main result, \cref{thm:main}, using the results of \cref{sec:kprop}.

The proofs of \cref{prop:FOLkprop}, \cref{prop:FOLLkprop}, and \cref{prop:SOkprop} rely on sophisticated estimates for certain singular integrals, which we have deferred to \cref{sec:app}. Since harmonic analysts are not the target audience of this article, we include in this appendix a review of multilinear singular integral operators/forms, in particular the important aforementioned work of Christ and Journ\'{e} \cite{CJ1987}.

\subsection{Acknowledgments}
\label{ssec:intro_ack}
The author thanks Gigliola Staffilani for her helpful discussion and encouragement during the early stages of this work. The author also thanks Fanghua Lin for his correspondence, which has improved the presentation of this article. Finally, the author gratefully acknowledges financial support from the Simons Collaboration on Wave Turbulence.

\section{Preliminaries}
\label{sec:pre}

\subsection{Basic Notation}
\label{ssec:pre_not}
Given nonnegative quantities $A$ and $B$, we write $A\lesssim B$ if there exists a constant $C>0$, independent of $A$ and $B$, such that $A\leq CB$. If $A \lesssim B$ and $B\lesssim A$, we write $A\sim B$. To emphasize the dependence of the constant $C$ on some parameter $p$, we sometimes write $A\lesssim_p B$ or $A\sim_p B$.

We denote the natural numbers excluding zero by $\N$ and including zero by $\N_0$. Similarly, we denote the nonnegative real numbers by $\R_{\geq 0}$ and the positive real numbers by $\R_+$ or $\R_{>0}$.

Given $N\in\N$ and points $x_{1,N},\ldots,x_{N,N}$ in some set $X$, we will write $\ux_N$ to denote the $N$-tuple $(x_{1,N},\ldots,x_{N,N})$. We define the generalized diagonal $\Delta_N$ of the Cartesian product $X^N$ to be the set
\begin{equation}
\Delta_N \coloneqq \{(x_1,\ldots,x_N) \in X^N : x_i=x_j \text{ for some $i\neq j$}\}.
\end{equation}
Given $x\in\R^n$ and $r>0$, we denote the ball and sphere centered at $x$ of radius $r$ by $B(x,r)$ and $\p B(x,r)$, respectively. We denote the uniform probability measure on the sphere $\p B(x,r)$ by $\sigma_{\p B(x,r)}$. Given a function $f$, we denote the support of $f$ by $\supp f$. We use the notation $\jp{x}\coloneqq (1+|x|^2)^{1/2}$ to denote the Japanese bracket.

If $A=(A^{ij})_{i,j=1}^N$ and $B=(B^{ij})_{i,j=1}^N$ are two $N\times N$ matrices, with entries in $\C$, we denote their Frobenius inner product by
\begin{equation}
A : B \coloneqq \sum_{i,j=1}^N A^{ij}\ol{B^{ij}}.
\end{equation}

We denote the space of complex-valued Borel measures on $\R^n$ by $\M(\R^n)$. We denote the subspace of probability measures (i.e. elements $\mu\in\M(\R^n)$ with $\mu\geq 0$ and $\mu(\R^n)=1$) by $\P(\R^n)$. When $\mu$ is in fact absolutely continuous with respect to Lebesgue measure on $\R^n$, we shall abuse notation by writing $\mu$ for both the measure and its density function. 

We denote the Banach space of complex-valued continuous, bounded functions on $\R^n$ by $C(\R^n)$ equipped with the uniform norm $\|\cdot\|_{\infty}$. More generally, we denote the Banach space of $k$-times continuously differentiable functions with bounded derivatives up to order $k$ by $C^k(\R^n)$ equipped with the natural norm, and we define $C^\infty \coloneqq \bigcap_{k=1}^\infty C^k$. We denote the subspace of smooth functions with compact support by $C_c^\infty(\R^n)$, and use the subscript $0$ to indicate functions vanishing at infinity. We denote the Schwartz space of functions by $\Sc(\R^n)$ and the space of tempered distributions by $\Sc'(\R^n)$.

For $p\in [1,\infty]$ and $\Omega\subset\R^n$, we define $L^p(\Omega)$ to be the usual Banach space equipped with the norm
\begin{equation}
\|f\|_{L^p(\Omega)} \coloneqq \paren*{\int_\Omega |f(x)|^p dx}^{1/p}
\end{equation}
with the obvious modification if $p=\infty$. When $f: \Omega\rightarrow X$ takes values in some Banach space $(X,\|\cdot\|_{X})$, we shall write $\|f\|_{L^p(\Omega;X)}$.

Our conventions for the Fourier transform and inverse Fourier transform are respectively
\begin{align}
\mathcal{F}(f)(\xi) &\coloneqq \wh{f}(\xi) \coloneqq \int_{\R^n}f(x)e^{-i\xi\cdot x}dx \qquad \forall \xi\in\R^n,\\
\mathcal{F}^{-1}(f)(x) &\coloneqq f^{\vee}(x) \coloneqq \frac{1}{(2\pi)^n}\int_{\R^n}f(\xi)e^{i x\cdot\xi}d\xi \qquad \forall x\in\R^n.
\end{align}

For integer $k\in\N_0$ and $1\leq p\leq \infty$, we define the usual Sobolev spaces
\begin{equation}
\begin{split}
W^{k,p}(\R^n) \coloneqq \{\mu \in L^p(\R^n) : \nabla^k \mu \in L^p(\R^n;(\C^n)^{\otimes k}),\quad \|\mu\|_{W^{k,p}(\R^n)} &\coloneqq \sum_{k=0}^n \|\nabla^{k}\mu\|_{L^p(\R^n)}.
\end{split}
\end{equation}
For $s\in\R$, we define the inhomogeneous Sobolev space $H^s(\R^n)$ to be the space of $\mu\in\Sc'(\R^n)$ such that $\wh{\mu}$ is locally integrable and 
\begin{equation}
\label{eq:H^s_def}
\|\mu\|_{H^s(\R^n)} \coloneqq \paren*{\int_{\R^n} \jp{\xi}^{2s} |\wh{\mu}(\xi)|^2d\xi }^{1/2}<\infty,
\end{equation}
and we use the notation $\|\mu\|_{\dot{H}^s(\R^n)}$ to denote the semi-norm where $\jp{\xi}$ is replaced by $|\xi|$.

\subsection{Harmonic Analysis}
\label{ssec:pre_HA}
In this subsection, we review some facts about function spaces, Littlewood-Paley theory, and Riesz potential estimates. This material is standard in the field, and the reader can consult harmonic analysis references such as \cite{Stein1970,Stein1993,grafakos2014c, grafakos2014m}.

\begin{mydef}[Riesz potential]
\label{def:RP}
Let $n\in\N$. For $s>-n$, we define the Fourier multiplier $(-\Delta)^{s/2}$ by
\begin{equation}
((-\Delta)^{s/2}f)(x) \coloneqq (|\cdot|^{s}\wh{f}(\cdot))^\vee(x) \qquad x\in \R^n,
\end{equation}
for a Schwartz function $f\in \Sc(\R^n)$. Since, for $s\in (-n,0)$, the inverse Fourier transform of $|\xi|^s$ is the tempered distribution
\begin{equation}
\frac{2^s\Gamma(\frac{n+s}{2})}{\pi^{\frac{n}{2}}\Gamma(-\frac{s}{2})} |x|^{-s-n},
\end{equation}
it follows that
\begin{equation}
((-\Delta)^{s/2}f)(x) = \frac{2^s \Gamma(\frac{n+s}{2})}{\pi^{\frac{n}{2}}\Gamma(-\frac{s}{2})}\int_{\R^n}\frac{f(y)}{|x-y|^{s+n}}dy, \qquad x\in\R^n.
\end{equation}
For $s\in (0,n)$, we define the \emph{Riesz potential operator} of order $s$ by $\mathcal{I}_s \coloneqq (-\Delta)^{-s/2}$ on $\Sc(\R^n)$.
\end{mydef}

$\mathcal{I}_s$ extends to a well-defined operator on any $L^p$ space, the extension also denoted by $\mathcal{I}_s$ with an abuse notation, as a consequence of the \emph{Hardy-Littlewood-Sobolev (HLS) lemma}.

\begin{prop}[Hardy-Littlewood-Sobolev]
\label{prop:HLS}
Let $n\in\N$, $s \in (0,n)$, and $1<p<q<\infty$ satisfy the relation
\begin{equation}
\frac{1}{p}-\frac{1}{q} = \frac{s}{n}.
\end{equation}
Then for all $f\in\Sc(\R^n)$,
\begin{align}
\|\mathcal{I}_s(f)\|_{L^q(\R^n)} &\lesssim_{n,s,p} \|f\|_{L^p(\R^n)},\\
\|\mathcal{I}_s(f)\|_{L^{\frac{n}{n-s},\infty}(\R^n)} &\lesssim_{n,s} \|f\|_{L^1(\R^n)},
\end{align}
where $L^{r,\infty}$ denotes the weak-$L^r$ space. Consequently, $\mathcal{I}_s$ has a unique extension to $L^p$, for all $1\leq p<\infty$.
\end{prop}

The next lemma allows us to control the $L^\infty$ norm of $\mathcal{I}_s(f)$ in terms of the $L^1$ norm and $L^p$ norm, for some $p=p(s,n)$.
\begin{lemma}[$L^{\infty}$ bound for Riesz potential]
\label{lem:Linf_RP}
For any $n\in\N$, $s\in (0,n)$, and $p\in (\frac{n}{s},\infty]$,
\begin{equation}
\|\mathcal{I}_s(f)\|_{L^\infty(\R^n)} \lesssim_{s,n,p} \|f\|_{L^{1}(\R^n)}^{1-\frac{n-s}{n(1-\frac{1}{p})}} \|f\|_{L^{p}(\R^n)}^{\frac{n-s}{n(1-\frac{1}{p})}}.
\end{equation}
\end{lemma}

To introduce the Besov scale of function spaces, we must first recall the rudiments of Littlewood-Paley theory. Let $\phi\in C_{c}^{\infty}(\R^n)$ be a radial, nonincreasing function, such that $0\leq \phi\leq 1$ and
\begin{equation}
\phi(x)=
\begin{cases}
1, & {|x|\leq 1}\\
0, & {|x|>2}
\end{cases}.
\end{equation}
Define the dyadic partitions of unity
\begin{align}
1&=\phi(x)+\sum_{j=1}^{\infty}[\phi(2^{-j}x)-\phi(2^{-j+1}x)] \eqqcolon \psi_{\leq 0}(x)+\sum_{j=1}^{\infty}\psi_{j}(x) \qquad \forall x\in\R^n,\\
1&=\sum_{j\in\mathbb{Z}}[\phi(2^{-j}x)-\phi(2^{-j+1}x)] \eqqcolon \sum_{j\in\mathbb{Z}}{\psi}_{j}(x) \qquad \forall x\in\R^n\setminus\{0\}.
\end{align}
For any $j\in\Z$, we define the Littlewood-Paley projector $P_{j}, P_{\leq 0}$ by
\begin{align}
(P_jf)(x) &\coloneqq (\psi_{j}(D)f)(x) = \int_{\R^n}K_{j}(x-y)f(y)dy \qquad K_j \coloneqq \psi_j^{\vee},\\
(P_{\leq 0}f)(x) &\coloneqq (\psi_{\leq 0}(D)f)(x) = \int_{\R^n} K_{\leq 0}(x-y)f(y)dy \qquad K_{\leq 0} \coloneqq \psi_{\leq 0}^{\vee}.
\end{align}

\begin{mydef}[Besov space]
\label{def:Bes}
Let $s\in\R$ and $1\leq p,q\leq\infty$. We define the inhomogeneous Besov space $B_{p,q}^s(\R^n)$ to be the space of $\mu\in\Sc'(\R^n)$ such that
\begin{equation}
\|\mu\|_{B_{p,q}^s(\R^n)} \coloneqq \paren*{\|P_{\leq 0}\mu\|_{L^p(\R^n)} + \sum_{j=1}^\infty 2^{jqs}\|P_j\mu\|_{L^p(\R^n)}^q}^{1/q} < \infty.
\end{equation}
For $p,q,s$ as above, we also define the homogeneous Besov semi-norm
\begin{equation}
\|\mu\|_{\dot{B}_{p,q}^s(\R^n)} \coloneqq \paren*{\sum_{j\in\Z} 2^{jqs} \|P_j\mu\|_{L^p(\R^n)}^{q}}^{1/q}.
\end{equation}
\end{mydef}

\begin{remark}
Any two choices of Littlewood-Paley partitions of unity used to define $\|\cdot\|_{B_{p,q}^s(\R^n)}$ (resp. $\|\cdot\|_{\dot{B}_{p,q}^s(\R^n)}$) lead to equivalent norms (resp. semi-norms). 
\end{remark}

\begin{remark}
From Plancherel's theorem, we see that the space $B_{2,2}^s(\R^n)$ coincides with the Sobolev space $H^s(\R^n)$. For $s\in \R_+\setminus\N$, the space $B_{\infty,\infty}^s(\R^n)$ coincides with the H\"older space $C^{[s],s-[s]}(\R^n)$ of bounded functions $\mu:\R^n\rightarrow\C$ such that $\nabla^{k}\mu$ is bounded, for integers $0\leq k\leq [s]$ and
\begin{equation}
\|\nabla^{[s]}\mu\|_{\dot{C}^{s-[s]}(\R^n)} \coloneqq \sup_{{0<|x-y|\leq 1}} \frac{|(\nabla^{[s]}\mu)(x) - (\nabla^{[s]}\mu)(y)|}{|x-y|^{s-[s]}} <\infty. 
\end{equation}
For integer $s$, $B_{\infty,\infty}^s(\R^n)$ is the Zygmund space of order $s$, which is \emph{strictly larger} than $C^s(\R^n)$.
\end{remark}

We next define the space of \emph{log-Lipschitz functions}, which is quite relevant given that the Biot-Savart velocity $u$ is log-Lipschitz, but not Lipschitz, for $L^\infty$ vorticity.
\begin{mydef}[Log-Lipschitz space]
\label{def:LL}
We define $LL(\R^n)$ to be the space of functions $\mu\in C(\R^n)$ such that
\begin{equation}
\|\mu\|_{LL(\R^n)} \coloneqq \sup_{0<|x-y|\leq e^{-1}} \frac{|\mu(x)-\mu(y)|}{|x-y||\ln|x-y||} <\infty.
\end{equation}
\end{mydef}

\begin{lemma}
\label{lem:LL_embed}
It holds that
\begin{equation}
\|\mu\|_{LL(\R^n)} \lesssim_n \|\nabla \mu\|_{B_{\infty,\infty}^0(\R^n)}, \qquad \forall \mu\in B_{\infty,\infty}^1(\R^n).
\end{equation}
Consequently, $B_{\infty,\infty}^1(\R^n)$ continuously embeds in $LL(\R^n)$.
\end{lemma} 

The next lemma for the 2D Coulomb potential $\g(x)=-\frac{1}{2\pi}\ln|x|$ is from \cite{Rosenzweig2020_PVMF} and gives useful estimates for the potential energy and Biot-Savart velocity of a measure $\mu$. In particular, the lemma shows that the modulated energy is well-defined under the assumptions in \eqref{eq:main_ass}.
\begin{lemma}[{\cite[Lemma 2.10]{Rosenzweig2020_PVMF}}]
\label{lem:PE_bnds}
Suppose that $\mu\in L^1(\R^2)\cap L^p(\R^2)$, for some $1<p\leq\infty$, such that $\int_{\R^2}\ln\jp{x}|\mu(x)|dx<\infty$. Then the convolution $\g\ast\mu$ is a well-defined continuous function, and we have the point-wise estimate
\begin{equation}
|(\g\ast\mu)(x)| \lesssim_p \jp{x}^{\frac{p-1}{p}}\ln(2\jp{x}) \|\mu\|_{L^p(\R^2)} + \int_{\R^2}\ln(2\jp{y})|\mu(y)|dy.
\end{equation}
If $1<p\leq 2$, then
\begin{equation}
\begin{split}
\|\g\ast\mu\|_{\dot{B}_{\infty,\infty}^{\frac{2p-2}{p}}(\R^2)} &\lesssim_{p} \|\mu\|_{L^p(\R^2)};
\end{split}
\end{equation}
and if $2<p\leq\infty$, then
\begin{equation}
\begin{split}
\|\nabla(\g\ast\mu)\|_{L^\infty(\R^2)} &\lesssim_{p} \|\mu\|_{L^1(\R^2)}^{1-\frac{p}{2(p-1)}} \|\mu\|_{L^p(\R^2)}^{\frac{p}{2(p-1)}},\\
\|\nabla(\g\ast\mu)\|_{\dot{B}_{\infty,\infty}^{\frac{p-2}{p}}(\R^2)} &\lesssim_{p} \|\mu\|_{L^p(\R^2)}.
\end{split}
\end{equation}
\end{lemma}

\subsection{The Osgood Lemma}
\label{ssec:pre_Os}
Following the presentation of \cite[Section 3.1]{BCD2011}, we recall some facts about moduli of continuity and the Osgood lemma, which is a generalization of the Gronwall-Bellman inequality.

\begin{mydef}[Modulus of continuity]
\label{def:mod_cont}
Let $a\in (0,1]$. A \emph{modulus of continuity} is an increasing, nonzero continuous function $\rho:[0,a]\rightarrow [0,\infty)$ such that $\rho(0)=0$. We say that a modulus of continuity satisfies the \emph{Osgood condition} or is an \emph{Osgood modulus of continuity}, if
\begin{equation}
\label{eq:Os_con}
\int_0^a \frac{dr}{\rho(r)} = \infty.
\end{equation}
\end{mydef}

\begin{lemma}[Osgood lemma {\cite[Lemma 3.4, Corollary 3.5]{BCD2011}}]
\label{lem:Os}
Fix $a\in (0,1]$. Let $f:[t_0,T]\rightarrow [0,a]$ be a measurable function, $\gamma:[t_0,T]\rightarrow [0,\infty)$ a locally integrable function, and $\rho:[0,a]\rightarrow [0,\infty)$ an Osgood modulus of continuity. Suppose that there exists a constant $c>0$ such that
\begin{equation}
f(t)\leq c + \int_{t_0}^t\gamma(t')\rho(f(t'))dt' \qquad \text{a.e} \ t \in [t_0,T].
\end{equation}
Define the function
\begin{equation}
\mathfrak{M}:(0,a] \rightarrow [0,\infty), \qquad \mathfrak{M}(x) \coloneqq \int_{x}^a\frac{dr}{\rho(r)}dr.
\end{equation}
Then $\mathfrak{M}$ is bijective, and if $t$ is such that $\int_{t_0}^t\gamma(t')dt' \leq \mathfrak{M}(c)$, it holds that
\begin{equation}
f(t) \leq \mathfrak{M}^{-1}\paren*{\mathfrak{M}(c)-\int_{t_0}^t\gamma(t')dt'}.
\end{equation}
\end{lemma}

\begin{remark}
\label{rem:Os_ex_log}
A highly relevant example of an Osgood modulus of continuity is the function
\begin{equation}
\rho: [0,e^{-1}] \rightarrow [0,\infty), \qquad \rho(r) \coloneqq r\ln(r^{-1}).
\end{equation}
Indeed, the anti-derivative of the reciprocal of $\rho$ is, up to an additive constant, $-\ln\ln(r^{-1})$, so the claim follows from the fundamental theorem of calculus. The reader may check that
\begin{equation}
\mathfrak{M}(x) = \ln\ln(x^{-1}) \qquad \text{and} \qquad \mathfrak{M}^{-1}(y) = e^{-e^{y}}.
\end{equation}
\end{remark}

\subsection{Stochastic Euler equation}
\label{ssec:pre_Eul}
We briefly review the results of \cite{BFM2016} concerning the well-posedness of the stochastic Euler equation \eqref{eq:SEul_S}. We closely follow the presentation of \cite[Section 2]{BFM2016}.

We recall from the introduction that we have a fixed stochastic basis $(\Omega, \F, (\F_t)_{t\geq 0}, \mathbf{P})$ satisfying the usual assumptions, and we have independent real Brownian motions $\{W^k\}_{k=1}^\infty$ defined on this probability space and adapted to the filtration. We assume that the vector fields $\{\sigma_k\}_{k=1}^\infty$ are smooth,\footnote{This smoothness is purely qualitative: none of our estimates require more than the Lipschitz semi-norm of $\sigma_k$.} divergence-free, and satisfy
\begin{equation}
\|\usig\|_{\ell_k^2 W_x^{1,\infty}(\N\times\R^2)}^2 \coloneqq \sum_{k=1}^\infty \|\sigma_k\|_{W^{1,\infty}(\R^2)}^2 < \infty.
\end{equation}
The mathematical interpretation of the SPDE \eqref{eq:SEul_S} is based on the formally equivalent \emph{It\^{o} formulation}
\begin{equation}
\label{eq:SEul_I}
\begin{cases}
\p_t\xi + u\cdot\nabla\xi + \sum_{k=1}^\infty \sigma_k\cdot\nabla\xi \dot{W}^k = \frac{1}{2}\sum_{k=1}^\infty \paren*{(\sigma_k\cdot\nabla)\sigma_k\cdot\nabla\xi + \sigma_k^{\otimes 2} : \nabla^2\xi} \\
u = \nabla^\perp\g\ast\xi\\
\xi|_{t=0} = \xi^0
\end{cases}.
\end{equation}

\begin{mydef}[Progressively measurable]
For $T>0$, let $\xi\in L^\infty(\Omega\times[0,T]\times\R^2)$. We say that $\xi$ is \emph{$(\F_t)$-weakly progressively measurable} if for every $f\in L^1(\R^2)$, the process
\begin{equation}
t\mapsto \ipp{\xi(t),f} = \int_{\R^2}f(x)d\xi(t,x) 
\end{equation}
is progressively measurable.
\end{mydef}

The notion of weak solution to \eqref{eq:SEul_I} considered in the statement of our main result \cref{thm:main} and in this article is the following.
\begin{mydef}[Weak solution]
\label{def:WS}
Let $\xi^0\in L^\infty(\R^2)$. A \emph{$L^\infty$-weak solution} to the SPDE \eqref{eq:SEul_I} with initial datum $\xi^0$ is a $(\F_t)$-weakly progressively measurable element $\xi \in L^\infty(\Omega\times[0,T]\times\R^2)$ such that for every test function $\varphi \in C_c^\infty(\R^2)$, with probability one it holds that
\begin{equation}
\begin{split}
\ipp{\xi(t),\varphi} &= \ipp{\xi_0,\varphi} + \int_0^t \ipp{\xi(r), u(r)\cdot\nabla\varphi}dr + \sum_{k=1}^\infty \int_0^t \ipp{\xi(r),\sigma_k\cdot\nabla\varphi}dW^k \\
&\ph - \frac{1}{2}\sum_{k=1}^\infty\int_0^t \paren*{\ipp{\xi(r), (\sigma_k\cdot\nabla)\sigma_k\cdot\nabla\varphi} - \ipp{\xi(r), \sigma_k^{\otimes 2}: \nabla^2\varphi}}dr, \qquad \forall t\in [0,T].
\end{split}
\end{equation}
\end{mydef}

For the deterministic Euler vorticity equation, sufficiently nice weak solutions are given by the pushforward of the initial data under the flow, which solves a nonlocal, nonlinear ODE. For the stochastic vorticiy equation \eqref{eq:SEul_I}, the analogous SDE (in integral form) is
\begin{equation}
\label{eq:SDE_Eul}
\Phi_t(x) = x + \int_0^t \int_{\R^2}\nabla^\perp\g(\Phi_s(x)-\Phi_s(y))d\xi^0(y)ds + \sum_{k=1}^\infty \int_0^t \sigma_k(\Phi_s(x))\circ dW_s^k.
\end{equation}
We now define the stochastic analogue of a flow.

\begin{mydef}[Stochastic flow]
A \emph{stochastic continuous flow} is a measurable map $\Phi: \Omega\times[0,T]\times\R^2 \rightarrow \R^2$ such that for almost every $\om\in\Omega$ fixed, the map $\Phi(\om): [0,T]\times\R^2\rightarrow\R^2$ is continuous and for every $x\in\R^2$ fixed, the map $\Phi(x):\Omega\times[0,T]\rightarrow\R^2$ is progressively measurable.

We say that a stochastic continuous flow $\Phi$ is \emph{measure-preserving} if, there exists a subset $\tl{\Omega}\subset \Omega$ of full probability such that for every $\om\in\tl{\Omega}$ and $t\in [0,T]$ fixed, the map $\Phi(t,\om):\R^2\rightarrow\R^2$ preserves Lebesgue measure.

A stochastic continuous flow $\Phi$ is a solution to the SDE \eqref{eq:SDE_Eul} if for every $x\in\R^2$, the process $X_t\coloneqq \Phi_t(x)$ is a solution to the SDE
\begin{equation}
\begin{cases}
dX_t = u^\Phi(t,X_t)dt + \sum_{k=1}^\infty \sigma_k(X_t)\circ W_t^k \\
u^{\Phi}(t,x) \coloneqq \int_{\R^2}(\nabla^\perp\g)(x-\Phi(t,y))d\xi^0(y) \\
X|_{t=0} = x
\end{cases}.
\end{equation}
\end{mydef}

Finally, we state the main well-posedness result of Brze\'{z}niak et al. \cite{BFM2016}. As remarked in the introduction, the authors of \cite{BFM2016} consider the periodic case; however, one can adapt their proofs to treat the case of $\R^2$, which is the setting of this article. Additionally, the authors of \cite{BFM2016} impose the assumption that
\begin{equation}
\sum_{k=1}^\infty \sigma_k^i(x)\sigma_k^j(x) = c\d_{ij} \qquad \forall i,j\in\{1,2\},
\end{equation}
for some $c\in\R$. This assumption has the benefit of eliminating the term $\sum_{k=1}^\infty (\sigma_k\cdot\nabla)\sigma_k\cdot\nabla\xi$ from equation \eqref{eq:SEul_I}, thereby simplifying the computations. But as the authors of \cite{BFM2016} note (see \cite[Remark 2.5]{BFM2016}), this assumption is not necessary; and we do not impose it in the present article.

\begin{thm}[Well-posedness]
\label{thm:GWP}
Let $T>0$. For $\xi^0\in \P(\R^2)\cap L^\infty(\R^2)$, there exists a unique solution
\begin{equation}
\xi\in L^\infty(\Omega\times[0,T]; \P(\R^2)) \cap L^\infty(\Omega\times[0,T]\times\R^2)
\end{equation}
to \eqref{eq:SEul_I} in the sense of \cref{def:WS}. Moreover, there exists a unique measure-preserving stochastic continuous flow solution to \eqref{eq:SDE_Eul}, which is $C_x^\alpha$ and $C_t^\beta$ for some $\al>0$ and every $\beta<1/2$, respectively. The weak solution $\xi$ is the pushforward of the initial datum under the flow $\Phi$:
\begin{equation}
\xi(t) = \Phi_t\#\xi^0.
\end{equation}
\end{thm}

\begin{remark}
\label{rem:con_LP}
Since the unique solution $\xi$ in \cref{thm:GWP} is the pushforward of the initial datum under the flow and the flow is a.s. measure-preserving for all time, it follows that the $L^p$ norms of $\xi$ are a.s. conserved on the interval $[0,T]$. 
\end{remark}

\section{The Modulated Energy}
\label{sec:CE}
We briefly review the modulated energy and its renormalization, which measure the distance between the $N$-body empirical measure $\xi_N$ and the mean-field measure $\xi$ in \cref{thm:main}. Since the results included in this section are to be found with proofs elsewhere in the literature, we include only statements with a reference to where the corresponding proofs are given.

\subsection{Setup}
\label{ssec:CE_set}
Recall that $\g(x)\coloneqq -\frac{1}{2\pi}\ln|x|$ is the 2D Coulomb potential. Following \cite{Duerinckx2016, Serfaty2017, Serfaty2020}, given $\eta>0$, we define the \emph{truncation to distance $\eta$} of $\g$ by
\begin{equation}
\label{eq:g_trunc}
\g_\eta:\R^2\rightarrow\R, \qquad \g_\eta(x) \coloneqq \begin{cases} \g(x), & |x|\geq \eta \\ \tl{\g}(\eta), & |x|<\eta \end{cases},
\end{equation}
where we have introduced the notation $\g(x) = \tl{\g}(|x|)$ to reflect that $\g$ is spherically symmetric. Evidently, $\g_\eta$ is continuous on $\R^2$ and decreases like $\g$ as $|x|\rightarrow\infty$.

\begin{lemma}[{\cite[Lemma 3.1]{Rosenzweig2020_PVMF}}]
\label{lem:g_id}
For any $\eta>0$, we have the distributional identities
\begin{align}
(\nabla\g_\eta)(x) &= -\frac{x}{2\pi|x|^2}1_{\geq \eta}(x), \label{eq:g_eta_grad_id}\\
(\D\g_\eta)(x) &= -\sigma_{\p B(0,\eta)}(x), \label{eq:g_eta_lapl_id}
\end{align}
where $\sigma_{\p B(0,\eta)}$ is the uniform probability measure on the sphere $\p B(0,\eta)$.
\end{lemma}
Using \cref{lem:g_id}, we define the \emph{smearing to scale $\eta$} of the Dirac mass $\delta_0$ by
\begin{equation}
\label{eq:delta_smear}
\d_0^{(\eta)} \coloneqq -\Delta \g_\eta = \sigma_{\p B(0,\eta)}.
\end{equation}
From the associativity and commutativity of convolution and the fact that $\g$ is a fundamental solution of the operator $-\D$, we see that $\d_0^{(\eta)}$ satisfies the identity
\begin{equation}
\label{eq:g_conv_smear}
(\g\ast \d_0^{(\eta)})(x) = \g_\eta(x).
\end{equation}

Next, given a probability measure $\mu$, such that $\int_{\R^2}|\ln|x|| d\mu(x)$, a vector $\ux_N \in (\R^2)^N$, and vector $\ue_N\in (\R_+)^N$, we introduce the notation
\begin{align}
H_{N,\ue_N}^{\mu,\ux_N} &\coloneqq \g\ast (\sum_{i=1}^N\d_{x_i}^{(\eta_i)}-N\mu) \label{eq:HN_trun_def},
\end{align}
where $\d_{x_i}^{(\eta_i)} = \d_{0}^{(\eta_i)}(\cdot-x_i)$. This compact notation will come in handy in \cref{sec:kprop}.

\subsection{Energy Functional}
\label{ssec:CE_EF}
For a vector $\ux_N\in (\R^2)^N$ and a measure $\mu\in \P(\R^2)\cap L^p(\R^2)$, for some $1<p\leq\infty$, which has the property
\begin{equation}
\int_{(\R^2)^2}\ln\jp{x-y}d\mu(x)d\mu(y) <\infty,
\end{equation}
for example $\mu\in L^\infty(\R^2)$ and $\supp\mu$ is compact, we define the functional
\begin{equation}
\label{eq:def_EN}
\Fr_N(\ux_N,\mu) \coloneqq \int_{(\R^2)^2\setminus \D_2}\g(x-y)d(\sum_{i=1}^N\d_{x_i}-N\mu)(x)d(\sum_{i=1}^N\d_{x_i}-N\mu)(y)
\end{equation}
where $\D_2\coloneqq \{(x,y)\in (\R^2)^2 : x=y\}$. Note that $N^2\Fr_N^{avg} = \Fr_N$. Our first lemma shows that the quantity $H_{N,\ue_N}^{\mu,\ux_N}$ defined in \eqref{eq:HN_trun_def} belongs to $\dot{H}^1(\R^2)$.

\begin{lemma}[{\cite[Lemma 3.4]{Rosenzweig2020_PVMF}}]
\label{lem:fin_en}
Fix $N\in\N$. Let $\mu\in L^p(\R^2)$, for some $2<p\leq \infty$, such that $\int_{\R^2}\ln\jp{x}|\mu(x)|dx<\infty$, and let $\ux_N\in (\R^2)^N\setminus\D_N$. Then for any $\ue_N\in (\R_+)^N$, we have the identity
\begin{equation}
\label{eq:renorm_nts}
\begin{split}
&\int_{(\R^2)^2}\g(x-y)d(N\mu-\sum_{i=1}^N\d_{x_i}^{(\eta_i)})(x)d(N\mu-\sum_{i=1}^N\d_{x_i}^{(\eta_i)})(y) =\int_{\R^2} |(\nabla H_{N,\ul{\eta}_N}^{\mu,\ux_N})(x)|^2dx,
\end{split}
\end{equation}
In particular, the right-hand side is finite if and only if $\mu$ has finite Coulomb energy.
\end{lemma}

The next proposition is essentially proven in \cite[Section 2.1]{PS2017} and \cite[Section 5]{Serfaty2020} in the greater generality of Riesz, not just Coulomb, interactions. The version presented below is from \cite[Proposition 3.5]{Rosenzweig2020_PVMF}.
\begin{prop}[{\cite[Proposition 3.5]{Rosenzweig2020_PVMF}}]
\label{prop:CE}
Let $\mu \in \P(\R^2)\cap L^p(\R^2)$, for some $2<p\leq\infty$, such that $\int_{\R^2}\ln\jp{x} |\mu(x)|dx<\infty$, and let $\ux_N\in (\R^2)^N\setminus\D_N$. Then
\begin{equation}
\label{eq:EN_renorm_lim}
\Fr_N(\ux_N,\mu) = \lim_{|\ul{\eta}_N|\rightarrow 0} \paren*{\int_{\R^2}|(\nabla H_{N,\ue_N}^{\mu,\ux_N})(x)|^2dx - \sum_{i=1}^N\tl{\g}(\eta_i)}
\end{equation}
and there exists a constant $C_p>0$, such that
\begin{equation}
\label{eq:EN_renorm_bnd}
\begin{split}
\sum_{1\leq i\neq j\leq N} \paren*{\g(x_i-x_j)-\tl{\g}(\eta_i)}_{+} &\leq \Fr_N(\ux_N,\mu) - \int_{\R^2}|(\nabla H_{N,\ul{\eta}_N}^{\mu,\ux_N})(x)|^2dx + \sum_{i=1}^N\tl{\g}(\eta_i) \\
&\ph + C_pN \|\mu\|_{L^p(\R^2)}\sum_{i=1}^N\eta_i^{2(p-1)/p},
\end{split}
\end{equation}
where $(\cdot)_{+}\coloneqq \max\{\cdot,0\}$.
\end{prop}

The following corollary of \cref{prop:CE} relaxes the $\mu\in L^\infty(\R^2)$ assumption in \cite[Corollary 3.4]{Serfaty2020} and incorporates an additional parameter $\ep_1$. The version below is from \cite[Corollary 3.6]{Rosenzweig2020_PVMF}.
\begin{cor}[{\cite[Corollary 3.6]{Rosenzweig2020_PVMF}}]
\label{cor:grad_H}
Fix $N\in\N$. Let $\mu\in L^p(\R^2)$, for some $2<p\leq \infty$, such that $\int_{\R^2}\ln\jp{x}|\mu(x)|dx<\infty$, and let $\ux_N\in (\R^2)^N\setminus \D_N$. If for any $0<\ep_1 \ll 1$, we define
\begin{equation}
r_{i,\ep_1} \coloneqq \min\{\frac{1}{4}\min_{{1\leq j\leq N}\atop {j\neq i}} |x_i-x_j|, \ep_1\} \quad \text{and} \quad  \ur_{N,\ep_1}\coloneqq (r_{1,\ep_1},\ldots,r_{N,\ep_1}),
\end{equation}
then there exists a constant $C_p>0$ such that
\begin{equation}
\label{eq:g_r_sum_bnd}
\sum_{i=1}^N \tl{\g}(r_{i,\ep_1}) \leq \Fr_N(\ux_N,\mu) + 2N\tl{\g}(\ep_1) + C_pN^2\|\mu\|_{L^p(\R^2)}\ep_1^{\frac{2(p-1)}{p}}
\end{equation}
and
\begin{equation}
\label{eq:grad_H_r_bnd}
\int_{\R^2}|(\nabla H_{N,\ul{r}_{N,\ep_1}}^{\mu,\ux_N})(x)|^2dx \leq \Fr_N(\ux_N,\mu) + N\tl{\g}(\ep_1) + C_pN^2\|\mu\|_{L^p(\R^2)}\ep_1^{\frac{2(p-1)}{p}}.
\end{equation}
\end{cor}

The final result of this subsection is a lemma which uses the modulated energy $\Fr_N(\ux_N,\mu)$ to count the number of distinct pairs $(i,j)$, such that the distance between the particles $x_i$ and $x_j$ is below a prescribed threshold. The lemma presented below is from \cite[Lemma 3.7]{Rosenzweig2020_PVMF}.
\begin{lemma}[{\cite[Lemma 3.7]{Rosenzweig2020_PVMF}}]
\label{lem:count}
Fix $N\in\N$. Then there exists a constant $C_p>0$, such that for any $\ux_N\in (\R^2)^N\setminus\D_N$ and $\mu\in\P(\R^2)\cap L^p(\R^2)$, for some $2<p\leq\infty$, with finite Coulomb energy and such that $\int_{\R^2}\ln\jp{x}|\mu(x)|dx<\infty$, we have the cardinality bound
\begin{equation}
\begin{split}
\left|\{(i,j)\in\N^2: i\neq j \enspace \text{and} \enspace |x_i-x_j|\leq  \ep_3\}\right| \lesssim \Fr_N(\ux_N,\mu) + N\tl{\g}(\ep_3) + C_pN^2 \|\mu\|_{L^p(\R^2)} \ep_3^{\frac{2(p-1)}{p}},
\end{split}
\end{equation}
for any $0<\ep_3\ll 1$. 
\end{lemma}

\subsection{Coerciveness of the Energy}
\label{ssec:CE_coer}
In this final subsection of \cref{sec:CE}, we record a proposition showing that the functional $\Fr_N(\ux_N,\mu)$ controls convergence in the weak-* topology for the Besov space $B_{2,\infty}^{-1}(\R^2)$, as $N\rightarrow\infty$. For this Besov scale, $B_{2,\infty}^{-1}(\R^2)$ is the endpoint space containing the Dirac mass. The proposition also establishes the coerciveness of the normalized modulated energy $\Fr_N^{avg}(\ux_N,\mu)$, in the sense that it controls convergence in the weak-* topology on $\M(\R^2)$ as $N\rightarrow\infty$.

\begin{prop}[{\cite[Proposition 3.10]{Rosenzweig2020_PVMF}}]
\label{prop:bes_conv}
Let $N\in\N$ and $\ux_N\in (\R^2)^N\setminus\D_N$. Then for any $\mu\in \P(\R^2)\cap L^{p}(\R^2)$, for some $2< p\leq \infty$, and $\varphi\in B_{2,1}^{1}(\R^2)$, we have the estimate
\begin{equation}
\label{eq:B2inf_conv}
\begin{split}
\left|\int_{\R^2}\varphi(x)d(\sum_{i=1}^N \d_{x_i}-N\mu)(x)\right| &\lesssim  N\paren*{\frac{\ep_1\|\varphi\|_{B_{2,1}^1(\R^2)}}{\ep_2} + \sum_{k\geq |\log_2\ep_2|} 2^k\|P_k\varphi\|_{L^2(\R^2)}} \\
&\ph +\|\nabla\varphi\|_{L^2(\R^2)}\paren*{\Fr_N(\ux_N,\mu) + N|\ln \ep_1| + C_p\|\mu\|_{L^p(\R^2)} N^2 \ep_1^{\frac{2(p-1)}{p}}}^{1/2},
\end{split}
\end{equation}
for any parameters $0<\ep_1<\ep_2\ll 1$. Consequently, for any $s<-1$, 
\begin{align}
\|\mu-\frac{1}{N}\sum_{i=1}^N\d_{x_i}\|_{H^s(\R^2)} \lesssim_{s,p} |\Fr_N^{avg}(\ux_N,\mu)|^{1/2} + N^{-1/2}|\ln N|^{1/2} + \paren*{1+\|\mu\|_{L^p(\R^2)}}N^{-1/2}, \label{eq:Hs_conv}
\end{align}
and if $\Fr_N^{avg}(\ux_N,\mu)\rightarrow 0$, as $N\rightarrow\infty$, then
\begin{equation}
\frac{1}{N}\sum_{i=1}^N \d_{x_i}\xrightharpoonup[N\rightarrow\infty]{*} \mu \ \text{in $\M(\R^2)$} \label{eq:M_conv}.
\end{equation}
\end{prop}

\section{Key Propositions}
\label{sec:kprop}
This section is devoted to the proofs of the following propositions, which are the workhorses of this article.

\begin{restatable}{prop}{FOLkprop}
\label{prop:FOLkprop}
Assume that $\mu\in \P(\R^2)\cap L^p(\R^2)$ for some $2<p\leq\infty$. Then for any Lipschitz vector field $v:\R^2\rightarrow\R^2$ and vector $\ux_N\in (\R^2)^N\setminus\D_N$, we have the estimate
\begin{equation}
\label{eq:kprop_FOL}
\begin{split}
&\frac{1}{N^2}\left|\int_{(\R^2)^2\setminus\D_2}\nabla\g(x-y)\cdot\paren*{v(x)-v(y)}d(N\mu-\sum_{i=1}^N\d_{x_i})(x)d(N\mu-\sum_{i=1}^N\d_{x_i})(y) \right| \\
&\lesssim \|\nabla v\|_{L^\infty(\R^2)}\paren*{\Fr_N^{avg}(\ux_N,\mu) + \frac{|\ln\ep_3|}{N} + C_p\|\mu\|_{L^p(\R^2)}\ep_3^{\frac{2(p-1)}{p}}+ \ep_1 \paren*{C_p\|\mu\|_{L^p(\R^2)}^{\frac{p}{2(p-1)}} +\ep_3^{-1}}}.
\end{split}
\end{equation}
for all $(\ep_1,\ep_3)\in (\R_+)^2$ satisfying $0<4\ep_1<\ep_3\ll 1$. Here, $C_p$ is a constant depending only on $p$.
\end{restatable}

\begin{restatable}{prop}{FOLLkprop}
\label{prop:FOLLkprop}
Assume that $\mu\in \P(\R^2)\cap L^p(\R^2)$ for some $2<p\leq\infty$. Then for any log-Lipschitz vector field $v:\R^2\rightarrow\R^2$ and vector $\ux_N\in (\R^2)^N\setminus\D_N$, we have the estimate
\begin{equation}
\label{eq:kprop_FOLL}
\begin{split}
&\frac{1}{N^2}\left|\int_{(\R^2)^2\setminus\D_2}\nabla\g(x-y)\cdot\paren*{v(x)-v(y)}d(N\mu-\sum_{i=1}^N\d_{x_i})(x)d(N\mu-\sum_{i=1}^N\d_{x_i})(y) \right| \\
&\lesssim |\ln\ep_3|\|v\|_{LL(\R^2)}\paren*{\Fr_N^{avg}(\ux_N,\mu) + \frac{|\ln \ep_3|}{N} + C_p\|\mu\|_{L^p(\R^2)}\ep_3^{\frac{2(p-1)}{p}}} \\
&\ph + \|v\|_{LL(\R^2)}\ep_2|\ln\ep_2|\paren*{\ep_3^{-1} + C_p\|\mu\|_{L^p(\R^2)}^{\frac{p}{2(p-1)}}}
\end{split}
\end{equation}
for all $(\ep_2,\ep_3)\in (\R_+)^3$ satisfying $0<\ep_2<\ep_3\ll 1$. Here, $C_p$ is a constant depending only on $p$.
\end{restatable}

\begin{restatable}{prop}{SOkprop}
\label{prop:SOkprop}
Assume that $\mu\in \P(\R^2)\cap L^p(\R^2)$ for some $2<p\leq\infty$. Then for any Lipschitz vector field $v:\R^2\rightarrow\R^2$ and vector $\ux_N\in (\R^2)^N\setminus\D_N$, we have the estimate
\begin{equation}
\label{eq:kprop_SO}
\begin{split}
&\frac{1}{N^2}\left|\int_{(\R^2)^2\setminus \D_2} \nabla^2\g(x-y) : (v(x)-v(y))^{\otimes 2}d(N\mu-\sum_{i=1}^N\d_{x_i})(x)d(N\mu-\sum_{i=1}^N\d_{x_i})(y) \right| \\
&\lesssim \|\nabla v\|_{L^\infty(\R^2)}^2\paren*{\Fr_N^{avg}(\ux_N,\mu) + \frac{|\ln \ep_3|}{N}+ C_p\|\mu\|_{L^p(\R^2)}\ep_3^{\frac{2(p-1)}{p}} +\ep_1\paren*{C_p\|\mu\|_{L^p(\R^2)}^{\frac{p}{2(p-1)}} + \ep_3^{-1}}}
\end{split}
\end{equation}
for all $(\ep_1,\ep_3)\in (\R_+)^2$ satisfying $0<4\ep_1<\ep_3\ll 1$. Here, $C_p$ is a constant depending on $p$.
\end{restatable}

A version of \cref{prop:FOLkprop} was proved by Serfaty \cite[Proposition 1.1]{Serfaty2020} and is the crucial ingredient for her proof of mean-field convergence of the deterministic point vortex model to the 2D Euler vorticity equation under the assumption that the limiting velocity field is spatially Lipschitz. Later, the author relaxed the Lipschitz assumption in this proposition--and therefore on the limiting velocity--by introducing a new mollification argument,\footnote{This mollification idea was inspired by earlier work \cite{Rosenzweig2019_LL} of the author on the mean-field limit of the 1D $\d$ Bose gas.} for which $\ep_2$ serves as the mollification parameter. \cref{prop:FOLLkprop} is a version of \cite[Proposition 1.6]{Rosenzweig2020_PVMF} in this prior work of the author. For a more extensive discussion of the original proofs behind these propositions, we refer the interested reader to \cite[Section 4.1]{Rosenzweig2020_PVMF} for a detailed overview of the main steps and comments on the challenges posed by non-Lipschitz vector fields. 

\cref{prop:SOkprop} is completely new to this work and, as explained in the introduction, is motivated by the second-order correction in the Stratonovich-to-It\^o conversion stemming from the nonzero quadratic variation of Brownian motion. The proof of this proposition is the main focus of this section and is given in \cref{ssec:kprop_SO} following the outline given in \cref{ssec:intro_RM} of the introduction to the article. We also present new, streamlined proofs of \Cref{prop:FOLkprop} and \Cref{prop:FOLLkprop} in \Cref{ssec:kpropFOL} and \Cref{ssec:kpropFOLL}, respectively, which reflect the harmonic-analysis perspective behind the proof of \cref{prop:SOkprop}. 

\subsection{Proof of \cref{prop:FOLkprop}}
\label{ssec:kpropFOL}
In this subsection, we prove \cref{prop:FOLkprop}. We start by introducing a parameter vector $\ue_N\in (\R_+)^N$ whose precise value shall be specified at the end. Out of convenience, we introduce the notation
\begin{equation}
T_{1,v}f(x) \coloneqq \int_{\R^2} K_{1,v}(x,y)f(y)dy, \qquad K_{1,v}(x,y)\coloneqq \nabla\g(x-y)\cdot \paren*{v(x)-v(y)}.
\end{equation}
Using this notation going forward, we decompose
\begin{equation}
\begin{split}
&\int_{(\R^2)^2\setminus\D_2} K_{1,v}(x,y) d(N\mu-\sum_{i=1}^N\d_{x_i})(x)d(N\mu-\sum_{i=1}^N\d_{x_i})(y) = \mathrm{Term}_1 + \mathrm{Term}_2 + \mathrm{Term}_3
\end{split}
\end{equation}
where
\begin{align}
\mathrm{Term}_1  &\coloneqq \int_{(\R^2)^2\setminus\D_2} K_{1,v}(x,y) d(N\mu-\sum_{i=1}^N\d_{x_i}^{(\et_i)})(x)d(N\mu-\sum_{i=1}^N\d_{x_i}^{(\et_i)})(y), \\
\mathrm{Term}_2 &\coloneqq -2\int_{(\R^2)^2\setminus\D_2} K_{1,v}(x,y) d(N\mu-\sum_{i=1}^N\d_{x_i}^{(\et_i)})(x)d(\sum_{i=1}^N\d_{x_i}-\d_{x_i}^{(\et_i)})(y), \\
\mathrm{Term}_3 &\coloneqq \sum_{1\leq i,j\leq N}\int_{(\R^2)^2\setminus\D_2}K_{1,v}(x,y) d(\d_{x_i}-\d_{x_i}^{(\et_i)})(x)d(\d_{x_j}-\d_{x_j}^{(\et_j)})(y).
\end{align}
Note that we implicitly used the symmetry of $K_{1,v}(x,y)$ under $x \leftrightarrow y$ in obtaining $\mathrm{Term}_2$. We now proceed to estimate $\mathrm{Term}_1$, $\mathrm{Term}_2$, and $\mathrm{Term}_3$ individually.

\begin{description}[leftmargin=*]
\item[Estimate for $\mathrm{Term}_1$]
Using the Fubini-Tonelli theorem, we see that
\begin{equation}
\label{eq:FOL_T1_RHS}
\mathrm{Term}_1 = \int_{\R^2} T_{1,v}(N\mu-\sum_{i=1}^N\d_{x_i}^{(\et_i)})(x)d(N\mu-\sum_{i=1}^N\d_{x_i}^{(\et_i)})(x),
\end{equation}
Integrating by parts once in both $x$ and $y$ and recalling the notation \eqref{eq:HN_trun_def}, we see that the right-hand side of \eqref{eq:FOL_T1_RHS} equals
\begin{equation}
\int_{\R^2} (\nabla H_{N,\ue_N}^{\mu,\ux_N})(x)\cdot (\nabla T_{1,v}\nabla)(\nabla H_{N,\ue_N}^{\mu,\ux_N})(x) dx.
\end{equation}
By \cref{prop:T_1v}, $\nabla T_{1,v}\nabla$ is bounded from $L^2(\R^2)$ to $L^2(\R^2;(\R^2)^2)$ with operator norm $\lesssim \|\nabla v\|_{L^\infty(\R^2)}$. Therefore, it follows from Cauchy-Schwarz that
\begin{equation}
\label{eq:FOL_T1_fin}
|\mathrm{Term}_1| \lesssim \|\nabla v\|_{L^\infty(\R^2)} \|\nabla H_{N,\ue_N}^{\mu,\ux_N}\|_{L^2(\R^2)}^2.
\end{equation}

\item[Estimate for $\mathrm{Term}_2$]
We first split
\begin{equation}
\begin{split}
-\mathrm{Term}_2 &= \underbrace{2N\sum_{i=1}^N\int_{(\R^2)^2} K_{1,v}(x,y)d\mu(x)d(\d_{x_i}-\d_{x_i}^{(\et_i)})(y)}_{\eqqcolon \mathrm{Term}_{2,1}} \\
&\ph -\underbrace{2\sum_{1\leq i, j\leq N} \int_{(\R^2)^2}K_{1,v}(x,y)d\d_{x_i}^{(\et_i)}(x)d(\d_{x_j}-\d_{x_j}^{(\et_j)})(y)}_{\eqqcolon\mathrm{Term}_{2,2}}.
\end{split}
\end{equation}
We now estimate $\mathrm{Term}_{2,1}$ and $\mathrm{Term}_{2,2}$ separately.

For $\mathrm{Term}_{2,1}$, we use the symmetry of $K_{1,v}(x,y)$ under $x\leftrightarrow y$ and Fubini-Tonelli to write
\begin{equation}
\mathrm{Term}_{2,1} = 2N\sum_{i=1}^N\int_{\R^2}(T_{1,v}\mu)(y)d(\d_{x_i}-\d_{x_i}^{(\et_i)})(y).
\end{equation}
Making a change of variable and using that $\d_{x_i}^{(\et_i)}$ is a probability measure, we see that for each $1\leq i\leq N$,
\begin{equation}
\int_{\R^2}(T_{1,v}\mu)(y)d(\d_{x_i}-\d_{x_i}^{(\et_i)})(y) = \int_{\R^2}\paren*{(T_{1,v}\mu)(x_i)-(T_v\mu)(x_i+\et_i y)}d\d_{0}^{(1)}(y).
\end{equation}
Since $T_{1,v}\mu\in W^{1,\infty}(\R^2)$ by \cref{lem:Linf_RP} with gradient bound
\begin{align}
\|\nabla T_{1,v}(\mu)\|_{L^\infty(\R^2)} \lesssim \|\nabla v\|_{L^\infty(\R^2)}\|\mathcal{I}_1(\mu)\|_{L^\infty(\R^2)} \lesssim_p \|\nabla v\|_{L^\infty(\R^2)}\|\mu\|_{L^p(\R^2)}^{\frac{p}{2(p-1)}},
\end{align}
we can apply the mean-value theorem to obtain
\begin{equation}
\left|(T_{1,v}\mu)(x_i)-(T_{1,v}\mu)(x_i+\et_i y)\right| \lesssim \et_i\|\nabla v\|_{L^\infty(\R^2)}\|\mu\|_{L^p(\R^2)}^{\frac{p}{2(p-1)}}, \qquad \forall y\in \supp(\d_0^{(1)}).
\end{equation}
It now follows after a little bookkeeping that
\begin{equation}
\label{eq:FOL_T21}
|\mathrm{Term}_{2,1}| \lesssim_p N\|\nabla v\|_{L^\infty(\R^2)}\|\mu\|_{L^p(\R^2)}^{\frac{p}{2(p-1)}}\sum_{i=1}^N \et_i.
\end{equation} 

For $\mathrm{Term}_{2,2}$, we split the sum over $1\leq i,j\leq N$ into ``close'' and ``far'' pairs of points:
\begin{equation}
\sum_{1\leq i, j\leq N} = \sum_{{1\leq i, j\leq N}\atop{|x_i-x_j|<\ep_3}} + \sum_{{1\leq i\neq j\leq N}\atop{|x_i-x_j|\geq \ep_3}},
\end{equation}
where $\ep_3$ is as in the statement of the proposition. For the sum over close pairs, we use the trivial bound
\begin{equation}
\label{eq:Kv_sup}
\sup_{{x,y\in\R^2}\atop{x\neq y}} |K_{1,v}(x,y)| \lesssim \min\{\|\nabla v\|_{L^\infty(\R^2)}, \frac{\|v\|_{L^\infty(\R^2)}}{|x-y|}\}
\end{equation}
to crudely estimate
\begin{equation}
\sum_{{1\leq i,j\leq N}\atop{|x_i-x_j|<\ep_3}} \left|\int_{(\R^2)^2} K_{1,v}(x,y)d\d_{x_i}^{(\et_i)}(x)d(\d_{x_j}-\d_{x_j}^{(\et_j)})(y)\right| \lesssim \|\nabla v\|_{L^\infty(\R^2)}|\{(i,j)\in\{1,\ldots,N\}^2 : |x_i-x_j| < \ep_3\}|
\end{equation}
and then apply \cref{lem:count} to the right-hand side to obtain
\begin{equation}
\label{eq:FOL_T22_cl}
\begin{split}
\sum_{{1\leq i, j\leq N}\atop{|x_i-x_j|<\ep_3}} \left|\int_{(\R^2)^2} K_{1,v}(x,y)d\d_{x_i}^{(\et_i)}(x)d(\d_{x_j}-\d_{x_j}^{(\et_j)})(y)\right| \lesssim \|\nabla v\|_{L^\infty(\R^2)}\paren*{\Fr_N(\ux_N,\mu) + N\tl{\g}(\ep_3) + C_pN^2 \ep_3^{\frac{2(p-1)}{p}} }.
\end{split}
\end{equation}
For the sum over far pairs, we first make a change of variable for $y$ to write
\begin{equation}
\int_{(\R^2)^2} K_{1,v}(x,y)d\d_{x_i}^{(\et_i)}(x)d(\d_{x_j}-\d_{x_j}^{(\et_j)})(y) = \int_{(\R^2)^2}\paren*{K_{1,v}(x,x_j)-K_v(x,x_j+\et_j y)}\d_{x_i}^{(\et_i)}(x)\d_{0}^{(1)}(y).
\end{equation}
Note that provided $\et_i\ll \ep_3$ for every $1\leq i\leq N$ and $|x_i-x_j|\geq \ep_3$, we may apply the mean-value theorem and (reverse) triangle inequality to obtain
\begin{align}
|K_{1,v}(x, x_j) - K_{1,v}(x, x_j + \et_j y)| &\leq |v(x)-v(x_j)|\left|\nabla\g(x-x_j)-\nabla\g(x-x_j-\et_j y)\right| \nn\\
&\ph + |\nabla\g(x-x_j-\et_j y)|\left|v(x_j+\et_j y)-v(x_j)\right| \nn\\
&\lesssim \frac{\|\nabla v\|_{L^\infty(\R^2)}\et_j}{\ep_3}
\end{align}
for every $(x,y)\in \supp(\d_{x_i}^{(\et_i)})\times\supp(\d_0^{(1)})$. Therefore,
\begin{equation}
\label{eq:FOL_T22_fr}
\sum_{{1\leq i\neq j\leq N}\atop{|x_i-x_j|\geq\ep_3}}\left|\int_{(\R^2)^2} K_{1,v}(x,y)d\d_{x_i}^{(\et_i)}(x)d(\d_{x_j}-\d_{x_j}^{(\et_j)})(y)\right| \lesssim N\frac{\|\nabla v\|_{L^\infty(\R^2)}}{\ep_3}\sum_{j=1}^N\et_j.
\end{equation}
Combining estimates \eqref{eq:FOL_T22_cl} and \eqref{eq:FOL_T22_fr}, we find that
\begin{equation}
\label{eq:FOL_T22}
|\mathrm{Term}_{2,2}| \lesssim \|\nabla v\|_{L^\infty(\R^2)}\paren*{\Fr_N(\ux_N,\mu) + N\tl{\g}(\ep_3) + C_pN^2 \ep_3^{\frac{2(p-1)}{p}} } + \frac{\|\nabla v\|_{L^\infty(\R^2)}N}{\ep_3}\sum_{j=1}^N\et_j.
\end{equation}

Now combining estimates \eqref{eq:FOL_T21} and \eqref{eq:FOL_T22}, we conclude that
\begin{equation}
\label{eq:FOL_T2_fin}
\begin{split}
|\mathrm{Term}_2| &\lesssim N\|\nabla v\|_{L^\infty(\R^2)}\paren*{C_p\|\mu\|_{L^p(\R^2)}^{\frac{p}{2(p-1)}} + \ep_3^{-1}}\sum_{i=1}^N \et_i \\
&\ph +\|\nabla v\|_{L^\infty(\R^2)}\paren*{\Fr_N(\ux_N,\mu) + N\tl{\g}(\ep_3) + C_pN^2 \ep_3^{\frac{2(p-1)}{p}} }.
\end{split}
\end{equation}

\item[Estimate for $\mathrm{Term}_3$]
As in our estimation of $\mathrm{Term}_2$, we decompose the sum over $i,j$ into sums over close and far vortex pairs with distance threshold $\ep_3$ to obtain
\begin{equation}
\mathrm{Term}_3 = \sum_{{1\leq i, j\leq N}\atop{|x_i-x_j|<\ep_3}}(\cdots) + \sum_{{1\leq i\neq j\leq N}\atop{|x_i-x_j|\geq\ep_3}}(\cdots) \eqqcolon \mathrm{Term}_{3,1} + \mathrm{Term}_{3,2}.
\end{equation}
We estimate $\mathrm{Term}_{3,1}$ and $\mathrm{Term}_{3,2}$ individually.

For $\mathrm{Term}_{3,1}$, we again use \eqref{eq:Kv_sup} to crudely estimate
\begin{equation}
\left|\int_{(\R^2)^2\setminus\D_2}K_{1,v}(x,y)d(\d_{x_i}-\d_{x_i}^{(\et_i)})(x)d(\d_{x_j}-\d_{x_j}^{(\et_j)})(y)\right| \lesssim \|\nabla v\|_{L^\infty(\R^2)}.
\end{equation}
So by \cref{lem:count},
\begin{align}
|\mathrm{Term}_{3,1}| &\lesssim \|\nabla v\|_{L^\infty(\R^2)}\sum_{{1\leq i, j\leq N}\atop{|x_i-x_j|<\ep_3}} 1 \nn\\
&\lesssim \|\nabla v\|_{L^\infty(\R^2)}\paren*{\Fr_N(\ux_N,\mu) + N\tl{\g}(\ep_3) + C_pN^2 \|\mu\|_{L^p(\R^2)} \ep_3^{\frac{2(p-1)}{p}}}. \label{eq:FOL_T31}
\end{align}
For $\mathrm{Term}_{3,2}$, the same mean-value-theorem reasoning used to obtain the estimate \eqref{eq:FOL_T22_fr} shows that
\begin{equation}
\label{eq:FOL_T32}
|\mathrm{Term}_{3,2}| \lesssim \sum_{{1\leq i\neq j\leq N}\atop{|x_i-x_j|\geq\ep_3}}\frac{\|\nabla v\|_{L^\infty(\R^2)}\et_j}{\ep_3} \leq  \frac{N\|\nabla v\|_{L^\infty(\R^2)}}{\ep_3}\sum_{j=1}^N \et_j
\end{equation} 

Now combining estimates \eqref{eq:FOL_T31} and \eqref{eq:FOL_T32}, we conclude that
\begin{equation}
\label{eq:FOL_T3_fin}
\begin{split}
|\mathrm{Term}_3| &\lesssim \frac{N\|\nabla v\|_{L^\infty(\R^2)}}{\ep_3}\sum_{j=1}^N \et_j  + \|\nabla v\|_{L^\infty(\R^2)}\paren*{\Fr_N(\ux_N,\mu) + N\tl{\g}(\ep_3) + C_pN^2 \|\mu\|_{L^p(\R^2)} \ep_3^{\frac{2(p-1)}{p}}}.
\end{split}
\end{equation}
\end{description}

Collecting our estimates \eqref{eq:FOL_T1_fin}, \eqref{eq:FOL_T2_fin}, and \eqref{eq:FOL_T3_fin} for $\mathrm{Term}_1$, $\mathrm{Term}_2$, and $\mathrm{Term}_3$, respectively, and simplifying, we find that there is a constant $C_p>0$ such that
\begin{equation}
\begin{split}
&\left|\int_{(\R^2)^2\setminus\D_2} K_{1,v}(x,y)d(N\mu-\sum_{i=1}^N\d_{x_i})(x)d(N\mu-\sum_{i=1}^N\d_{x_i})(y)\right| \\
&\lesssim \|\nabla v\|_{L^\infty(\R^2)} \|\nabla H_{N,\ue_N}^{\mu,\ux_N}\|_{L^2(\R^2)}^2  + NC_p\|\nabla v\|_{L^\infty(\R^2)}\|\mu\|_{L^p(\R^2)}^{\frac{p}{2(p-1)}}\sum_{i=1}^N \et_i \\
&\ph +\|\nabla v\|_{L^\infty(\R^2)}\paren*{\Fr_N(\ux_N,\mu) + N\tl{\g}(\ep_3) + C_pN^2\|\mu\|_{L^p(\R^2)} \ep_3^{\frac{2(p-1)}{p}} } + \frac{N\|\nabla v\|_{L^\infty(\R^2)}}{\ep_3}\sum_{j=1}^N\et_j .
\end{split}
\end{equation}
Choosing $\et_i = r_{i,\ep_1}\leq \ep_1$ for every $1\leq i\leq N$ and applying estimate \eqref{eq:grad_H_r_bnd} of \cref{cor:grad_H} together with the assumption $4\ep_1<\ep_3$, the preceding right-hand side is $\lesssim$
\begin{equation}
\begin{split}
&\|\nabla v\|_{L^\infty(\R^2)}\paren*{\Fr_N(\ux_N,\mu) + N\tl{\g}(\ep_3) + C_pN^2\|\mu\|_{L^p(\R^2)}\ep_3^{\frac{2(p-1)}{p}} + N^2\ep_1 \paren*{C_p\|\mu\|_{L^p(\R^2)}^{\frac{p}{2(p-1)}} +\ep_3^{-1}}}.
\end{split}
\end{equation}
Comparing this expression to the statement of \cref{prop:FOLkprop}, we see that the proof is complete.

\subsection{Proof of \cref{prop:FOLLkprop}}
\label{ssec:kpropFOLL}
In this subsection, we combine \cref{prop:FOLkprop} with the mollification argument from \cite{Rosenzweig2020_PVMF} in order to prove \cref{prop:FOLLkprop}. To this end, let $\chi\in C_c^\infty(\R^n)$ be a radial, nonincreasing bump function satisfying
\begin{equation}
\label{eq:chi}
\int_{\R^n}\chi(x) dx=1, \quad 0\leq \chi\leq 1, \quad \chi(x) = \begin{cases} 1, & {|x|\leq \frac{1}{4}} \\ 0, & {|x|>1} \end{cases}.
\end{equation}
For $\ep>0$, set
\begin{equation}
\label{eq:chi_v_ep}
\chi_\ep(x)\coloneqq \ep^{-2}\chi(x/\ep) \qquad \text{and} \qquad  v_\ep(x) \coloneqq (\chi_\ep\ast v)(x),
\end{equation}
where the convolution $\chi_{\ep}\ast v$ is performed component-wise. Evidently, $v_\vep$ is $C^\infty(\R^n;\R^n)$.
\begin{lemma}[{\cite[Lemma 2.11]{Rosenzweig2020_PVMF}}]
\label{lem:conv_bnds}
For $0<\ep\ll 1$ and $\mu\in LL(\R^n)$, we have the estimates
\begin{align}
\|\mu_\ep\|_{L^\infty(\R^n)} &\leq \|\mu\|_{L^\infty(\R^n)}, \label{eq:v_Linf} \\ 
\|\mu-\mu_\ep\|_{L^\infty(\R^n)} &\leq \|\mu\|_{LL(\R^n)}\ep|\ln\ep|, \label{eq:v_diff_Linf}\\
\|\nabla\mu_\ep\|_{L^\infty(\R^n)} &\lesssim_n \|\mu\|_{LL(\R^n)}|\ln\ep|. \label{eq:v_grad_Linf}
\end{align}
\end{lemma}

\begin{lemma}[{\cite[Lemma 4.2]{Rosenzweig2020_PVMF}}]
\label{lem:kprop_err}
There exists a constant $C_{p}>0$ such that for every $0<\ep_2,\ep_3\ll 1$, we have the estimate
\begin{equation}
\label{eq:LHS_split}
\begin{split}	
&\left|\int_{(\R^2)^2\setminus\D_2}K_{1,v-v_{\ep_2}}(x,y) d(\sum_{i=1}^N\d_{x_i}-N\mu)(x)d(\sum_{i=1}^N\d_{x_i}-N\mu)(y)\right| \\
&\lesssim \|v\|_{LL(\R^2)}|\ln \ep_3|\paren*{\Fr_N(\ux_N,\mu)+N|\ln\ep_3| + C_pN^2\|\mu\|_{L^p(\R^2)}\ep_3^{\frac{2(p-1)}{p}}} \\
&\ph + N^2\|v\|_{LL(\R^2)}\ep_2|\ln\ep_2|\paren*{\ep_3^{-1} + C_p\|\mu\|_{L^p(\R^2)}^{\frac{p}{2(p-1)}}}.
\end{split}
\end{equation}
\end{lemma}

The triangle inequality implies that
\begin{align}
&\left|\int_{(\R^2)^2\setminus\D_2} K_{1,v}(x,y)d(N\mu-\sum_{i=1}^N\d_{x_i})(x)d(N\mu-\sum_{i=1}^N\d_{x_i})(y)\right|\\
&\leq \left|\int_{(\R^2)^2\setminus\D_2} K_{1,v_{\ep_2}}(x,y)d(N\mu-\sum_{i=1}^N\d_{x_i})(x)d(N\mu-\sum_{i=1}^N\d_{x_i})(y)\right| \nn\\
&\ph + \left|\int_{(\R^2)^2\setminus\D_2} K_{1,v-v_{\ep_2}}(x,y)d(N\mu-\sum_{i=1}^N\d_{x_i})(x)d(N\mu-\sum_{i=1}^N\d_{x_i})(y)\right| \nn\\
&\eqqcolon \mathrm{Term}_1 + \mathrm{Term}_2.
\end{align}
Since $\|\nabla v_{\ep_2}\|_{L^\infty}\lesssim |\ln \ep_2|$ by \cref{lem:conv_bnds}, we may apply \cref{prop:FOLkprop} to $\mathrm{Term}_1$, obtaining
\begin{align}
|\mathrm{Term}_1| &\lesssim \|\nabla v_{\ep_2}\|_{L^\infty(\R^2)}\paren*{\Fr_N(\ux_N,\mu) + N\tl{\g}(\ep_1) + C_pN^2\|\mu\|_{L^p(\R^2)}\ep_1^{\frac{2(p-1)}{p}}} \nn\\
&\ph + C_pN^2\ep_1\|\nabla v_{\ep_2}\|_{L^\infty(\R^2)} \|\mu\|_{L^p(\R^2)}^{\frac{p}{2(p-1)}} + \frac{N^2\ep_1\|\nabla v_{\ep_2}\|_{L^\infty(\R^2)}}{\ep_3} \nn\\
&\lesssim |\ln\ep_2|\|v\|_{LL(\R^2)}\paren*{\Fr_N(\ux_N,\mu) + N|\ln \ep_1| + C_pN^2\|\mu\|_{L^p(\R^2)}\ep_1^{\frac{2(p-1)}{p}}} \nn\\
&\ph + C_p N^2\ep_1|\ln \ep_2| \|v\|_{LL(\R^2)}\|\mu\|_{L^p(\R^2)}^{\frac{p}{2(p-1)}} + \frac{N^2\ep_1 \|v\|_{LL(\R^2)}|\ln \ep_2|}{\ep_3}, \label{eq:FOLL_T1_fin}
\end{align}
for any choice $4\ep_1 < \ep_2$. Unpacking the definition of $K_{1,v-v_{\ep_2}}$ and applying \cref{lem:kprop_err} to $\mathrm{Term}_2$, we find that
\begin{equation}
\label{eq:FOLL_T2_fin}
\begin{split}
|\mathrm{Term}_2| &\lesssim \|v\|_{LL(\R^2)}|\ln \ep_3|\paren*{\Fr_N(\ux_N,\mu)+N|\ln\ep_3| + C_pN^2\|\mu\|_{L^p(\R^2)}\ep_3^{\frac{2(p-1)}{p}}} \\
&\ph + N^2\|v\|_{LL(\R^2)}\ep_2|\ln\ep_2|\paren*{\ep_3^{-1} + \|\mu\|_{L^p(\R^2)}^{\frac{p}{2(p-1)}}}.
\end{split}
\end{equation}
By combining estimates \eqref{eq:FOLL_T1_fin} and \eqref{eq:FOLL_T2_fin} for $\mathrm{Term}_1$ and $\mathrm{Term}_2$, respectively, we conclude the proof of \cref{prop:FOLLkprop}.

\subsection{Proof of \cref{prop:SOkprop}}
\label{ssec:kprop_SO}
In this subsection, we prove \cref{prop:SOkprop}. As before, let $\ue_N\in (\R_+)^N$ be a parameter vector to be optimized at the end. Similarly to with $K_{1,v}$, we use the notation
\begin{equation}
T_{2,v}f(x) \coloneqq \int_{\R^2} K_{2,v}(x,y)f(y)dy, \qquad K_{2,v}(x,y) \coloneqq \nabla^2\g(x-y) : (v(x)-v(y))^{\otimes 2}.
\end{equation}
Making a decomposition similar to at the beginning of the proof of \cref{prop:FOLkprop}, we see that
\begin{align}
&\int_{(\R^2)^2\setminus\D_2} K_{2,v}(x,y)d(N\mu-\sum_{i=1}^N\d_{x_i})(x)d(N\mu-\sum_{i=1}^N\d_{x_i})(y) = \mathrm{Term}_1 + \mathrm{Term}_2 + \mathrm{Term}_3,
\end{align}
where
\begin{align}
\mathrm{Term}_1 &\coloneqq \int_{(\R^2)^2\setminus\D_2} K_{2,v}(x,y)d(N\mu-\sum_{i=1}^N\d_{x_i}^{(\et_i)})(x)d(N\mu-\sum_{i=1}^N\d_{x_i}^{(\et_i)})(y), \\
\mathrm{Term}_2 &\coloneqq -2\int_{(\R^2)^2\setminus\D_2} K_{2,v}(x,y)d(N\mu-\sum_{i=1}^N\d_{x_i}^{(\et_i)})(x)d(\sum_{i=1}^N\d_{x_i}-\d_{x_i}^{(\et_i)})(y),\\
\mathrm{Term}_3 &\coloneqq \int_{(\R^2)^2\setminus\D_2} K_{2,v}(x,y)d(\sum_{i=1}^N\d_{x_i}-\d_{x_i}^{(\et_i)})(x)d(\sum_{i=1}^N\d_{x_i}-\d_{x_i}^{(\et_i)})(y).
\end{align}
We proceed to estimate each of the $\mathrm{Term}_j$ individually.

\begin{description}[leftmargin=*]
\item[Estimate for $\mathrm{Term}_1$]
Integrating by parts once in both $x$ and $y$, we find that
\begin{equation}
\mathrm{Term}_1 = \int_{\R^2}(\nabla H_{N,\ue_N}^{\mu,\ux_N})(x) \cdot (\nabla T_{2,v} \nabla)(\nabla H_{N,\ue_N}^{\mu,\ux_N})(x) dx.
\end{equation}
By Cauchy-Schwarz and \cref{prop:T_2v}, we conclude that
\begin{equation}
\label{eq:SO_T1_fin}
|\mathrm{Term}_1| \lesssim \|\nabla v\|_{L^\infty(\R^2)}^2 \|\nabla H_{N,\ue_N}^{\mu,\ux_N}\|_{L^2(\R^2)}^2.
\end{equation}

\item[Estimate for $\mathrm{Term}_2$]
We first write
\begin{equation}
\begin{split}
-\mathrm{Term}_2 &= \underbrace{2N\int_{(\R)^2}K_{2,v}(x,y) d\mu(x)d(\sum_{i=1}^N\d_{x_i}-\d_{x_i}^{(\et_i)})(y)}_{\eqqcolon \mathrm{Term}_{2,1}} \\
&\ph - \underbrace{2\sum_{1\leq i,j\leq N}\int_{(\R^2)^2} K_{2,v}(x,y)d\d_{x_i}^{(\et_i)}(x)d(\d_{x_j}-\d_{x_j}^{(\et_j)})(y)}_{\eqqcolon \mathrm{Term}_{2,2}}.
\end{split}
\end{equation}

For $\mathrm{Term}_{2,1}$, a change of variable implies that
\begin{equation}
\mathrm{Term}_{2,1} = 2N\int_{\R^2} \paren*{(T_{2,v}\mu)(x_i)-(T_{2,v}\mu)(x_i+y)}\d_{0}^{(\et_i)}(y).
\end{equation}
Since by \cref{lem:Linf_RP},
\begin{equation}
\|\nabla T_{2,v}\mu\|_{L^\infty(\R^2)} \lesssim \|\nabla v\|_{L^\infty(\R^2)}^2 \|\I_1\mu\|_{L^\infty(\R^2)} \lesssim_p \|\nabla v\|_{L^\infty(\R^2)}^2\|\mu\|_{L^p(\R^2)}^{\frac{p}{2(p-1)}},
\end{equation}
the mean-value theorem implies that
\begin{equation}
\label{eq:SO_T21_fin}
|\mathrm{Term}_{2,1}| \lesssim_p N\|\nabla v\|_{L^\infty(\R^2)}^2\|\mu\|_{L^p(\R^2)}^{\frac{p}{2(p-1)}}\sum_{i=1}^N \et_i.
\end{equation}

For $\mathrm{Term}_{2,2}$, we split the sum $\sum_{1\leq i\neq j\leq N}$ into ``close'' and ``far'' pairs $(i,j)$ with distance threshold $\ep_3$:
\begin{equation}
\sum_{1\leq i, j\leq N}(\cdots) = \sum_{{1\leq i, j\leq N}\atop{|x_i-x_j|<\ep_3}}(\cdots) + \sum_{{1\leq i\neq j\leq N}\atop{|x_i-x_j|\geq\ep_3}}(\cdots).
\end{equation}
We use the sup bound (cf. the bound \eqref{eq:Kv_sup})
\begin{equation}
\label{eq:Ksig_sup}
\sup_{{x,y\in\R^2}\atop{x\neq y}} |K_{2,v}(x,y)| \lesssim \|\nabla v\|_{L^\infty(\R^2)}^2
\end{equation}
together with \cref{lem:count} to bound the close-pair contribution by
\begin{equation}
\sum_{{1\leq i,j\leq N}\atop{|x_i-x_j|<\ep_3}} \|\nabla v\|_{L^\infty(\R^2)}^2 \lesssim \|\nabla v\|_{L^\infty(\R^2)}^2\paren*{\Fr_N(\ux_N,\mu) + N\tl{\g}(2\ep_3) + C_p N^2 \|\mu\|_{L^p(\R^2)}\ep_3^{\frac{2p-2}{p}}}.
\end{equation}
For the far-pair contribution, we use Fubini-Tonelli and make a change of variable in $y$ to write
\begin{equation}
\begin{split}
\int_{(\R^2)^2}K_{2,v}(x,y)d\d_{x_i}^{(\et_i)}(x)d(\d_{x_j}-\d_{x_j}^{(\et_j)})(y) &=\int_{(\R^2)^2} \paren*{K_{2,v}(x,x_j)-K_{2,v}(x,x_j+\et_j y)}d\d_{x_i}^{(\et_i)}(x)d\d_{0}^{(1)}(y).
\end{split}
\end{equation}
Provided that $\et_i\ll \ep_3$, the (reverse) triangle inequality and mean value theorem imply that for $|x_i-x_j|\geq \ep_3$,
\begin{align}
\left|K_{2,v}(x,x_j)-K_{2,v}(x,x_j+\et_j y)\right| &\lesssim \left|\nabla^2\g(x-x_j)-\nabla^2\g(x-x_j-\et_j y)\right| \left|(v(x)-v(x_j))^{\otimes 2}\right|  \nn\\
&\ph + \left|\nabla^2\g(x-x_j-\et_j)\right| \left|(v(x)-v(x_j))^{\otimes 2} - (v(x)-v(x_j+\et_j y))^{\otimes 2}\right| \nn\\
&\lesssim \frac{\et_j\|\nabla v\|_{L^\infty(\R^2)}^2}{\ep_3}
\end{align}
for all $(x,y)\in \supp(\d_{x_i}^{(\et_i)})\times \supp(\d_{0}^{(1)})$. Hence,
\begin{equation}
\begin{split}
&\sum_{{1\leq i\neq j\leq N}\atop{|x_i-x_j|\geq\ep_3}} \left|\int_{(\R^2)^2} \paren*{K_{2,v}(x,x_j)-K_{2,v}(x,x_j+\et_j y)}d\d_{x_i}^{(\et_i)}(x)d\d_{0}^{(1)}(y)\right| \lesssim \frac{N\|\nabla v\|_{L^\infty(\R^2)}^2}{\ep_3}\sum_{j=1}^N \et_j.
\end{split}
\end{equation}
After a little bookkeeping, we conclude that
\begin{equation}
\label{eq:SO_T22_fin}
\begin{split}
|\mathrm{Term}_{2,2}| &\lesssim \|\nabla v\|_{L^\infty(\R^2)}^2\paren*{\Fr_N(\ux_N,\mu) + N\tl{\g}(2\ep_3) + C_p N^2 \|\mu\|_{L^p(\R^2)}\ep_3^{\frac{2p-2}{p}} +\frac{N}{\ep_3}\sum_{i=1}^N \et_i}.
\end{split}
\end{equation}

Combining the estimates \eqref{eq:SO_T21_fin} and \eqref{eq:SO_T22_fin} for $\mathrm{Term}_{2,1}$ and $\mathrm{Term}_{2,2}$, respectively, we conclude that
\begin{equation}
\label{eq:SO_T2_fin}
\begin{split}
|\mathrm{Term}_2| &\lesssim \|\nabla v\|_{L^\infty(\R^2)}^2\paren*{N\|\mu\|_{L^p(\R^2)}^{\frac{p}{2(p-1)}}\sum_{i=1}^N \et_i + \Fr_N(\ux_N,\mu) + N\tl{\g}(2\ep_3) + C_p N^2 \|\mu\|_{L^p(\R^2)}\ep_3^{\frac{2p-2}{p}}} \\
&\ph +\frac{N\|\nabla v\|_{L^\infty(\R^2)}^2}{\ep_3}\sum_{i=1}^N \et_i.
\end{split}
\end{equation}

\item[Estimate for $\mathrm{Term}_3$]
We proceed similarly as in our estimation of $\mathrm{Term}_{3}$ in the proof of \cref{prop:FOLkprop}, leading us to conclude that
\begin{equation}
\label{eq:SO_T3_fin}
|\mathrm{Term}_3| \lesssim \|\nabla v\|_{L^\infty(\R^2)}^2\paren*{\Fr_N(\ux_N,\mu)+N\tl{\g}(\ep_3)+C_pN^2\|\mu\|_{L^p(\R^2)}\ep_3^{\frac{2(p-1)}{p}} + \frac{N}{\ep_3}\sum_{i=1}^N\et_i}.
\end{equation}
\end{description}

Combining the estimates \eqref{eq:SO_T1_fin}, \eqref{eq:SO_T2_fin}, and \eqref{eq:SO_T3_fin} for $\mathrm{Term}_1$, $\mathrm{Term}_2$, and $\mathrm{Term}_3$, respectively, and then simplifying, we conclude that
\begin{equation}
\begin{split}
&\left|\int_{(\R^2)^2\setminus\D_2} K_{2,v}(x,y)d(N\mu-\sum_{i=1}^N\d_{x_i})(x)d(N\mu-\sum_{i=1}^N\d_{x_i})(y)\right|\\
&\lesssim \|\nabla v\|_{L^\infty(\R^2)}^2 \|\nabla H_{N,\ue_N}^{\mu,\ux_N}\|_{L^2(\R^2)}^2  + \|\nabla v\|_{L^\infty(\R^2)}^2 N\paren*{C_p\|\mu\|_{L^p(\R^2)}^{\frac{p}{2(p-1)}} +\ep_3^{-1}}\sum_{i=1}^N \et_i \\
&\ph + \|\nabla v\|_{L^\infty(\R^2)}^2\paren*{\Fr_N(\ux_N,\mu)+N|\ln\ep_3| + C_pN^2\|\mu\|_{L^p(\R^2)}\ep_3^{\frac{2(p-1)}{p}}}.
\end{split}
\end{equation}
Again choosing $\et_i=r_{i,\ep_1}\leq \ep_1$ for $1\leq i\leq N$ and applying estimate \eqref{eq:grad_H_r_bnd} of \cref{cor:grad_H}, we obtain that the right-hand side of the preceding inequality is $\lesssim$
\begin{equation}
\begin{split}
&\|\nabla v\|_{L^\infty(\R^2)}^2\paren*{\Fr_N(\ux_N,\mu) + N|\ln \ep_3|+ C_pN^2\|\mu\|_{L^p(\R^2)}\ep_3^{\frac{2(p-1)}{p}}}\\
&\ph + \|\nabla v\|_{L^\infty(\R^2)}^2N^2 \ep_1(C_p\|\mu\|_{L^p(\R^2)}^{\frac{p}{2(p-1)}} + \ep_3^{-1}).
\end{split}
\end{equation}
Recalling the statement of \cref{prop:SOkprop}, we see that the proof is complete.

\section{Proof of Main Results}
\label{sec:MR}
In this last section, we prove our main result, \cref{thm:main}, using the results of \cref{sec:kprop}. We first record a lemma giving the It\^o equation satisfied by the modulated energy $\Fr_N^{avg}(\ux_N(t),\xi(t))$ viewed as a real-valued stochastic process. We leave filling in the details of the proof of the lemma as an exercise for the interested reader.

\begin{lemma}[ME derivative]
For $N\in\N$, let $\ux_N:\Omega\times[0,T]\rightarrow (\R^2)^N\setminus\D_2$ be a strong solution to the system \eqref{eq:SPVM}. Let $\xi\in L^\infty(\Omega\times[0,T]; \P(\R^2)\cap L^\infty(\R^2))$ be a weak solution to equation \eqref{eq:SEul_S} satisfying the condition \eqref{eq:main_ass}. Then $\Fr_N^{avg}(\ux_N,\xi): [0,T]\rightarrow \R$ is $C^{1/2-}$ and for any $\vep>0$, with probability one, we have the It\^o identity
\begin{equation}
\label{eq:MR_ME_deriv}
\begin{split}
&\jp{\Fr_N^{avg}(\ux_N(t),\xi(t))}_{\vep} - \jp{\Fr_N^{avg}(\ux_N(0),\xi(0))}_{\vep} \\
&= \int_0^t \frac{\Fr_N^{avg}(\ux_N(s),\xi(s))}{\jp{\Fr_N^{avg}(\ux_N(s),\xi(s))}_{\vep}}\int_{(\R^2)^2\setminus\D_2} K_{1,u}(x,y)d(\xi-\xi_N)(s,x)d(\xi-\xi_N)(s,y)ds \\
&\ph \frac{1}{2}\sum_{k=1}^\infty \int_0^t \frac{\Fr_N^{avg}(\ux_N(s),\xi(s))}{\jp{\Fr_N^{avg}(\ux_N(s),\xi(s))}_{\vep}}\int_{(\R^2)^2\setminus\D_2} K_{1,(\sigma_k\cdot\nabla)\sigma_k}(x,y)d(\xi-\xi_N)(s,x)d(\xi-\xi_N)(s,y)ds \\
&\ph + \sum_{k=1}^\infty \int_0^t \frac{\Fr_N^{avg}(\ux_N(s),\xi(s))}{\jp{\Fr_N^{avg}(\ux_N(s),\xi(s))}_{\vep}}\int_{(\R^2)^2\setminus\D_2} K_{1,\sigma_k}(x,y)d(\xi-\xi_N)(s,x)d(\xi-\xi_N)(s,y) dW^k(s) \\
&\ph + \frac{1}{2}\sum_{k=1}^\infty \int_0^t \frac{\Fr_N^{avg}(\ux_N(s),\xi(s))}{\jp{\Fr_N^{avg}(\ux_N(s),\xi(s))}_{\vep}} \int_{(\R^2)^2\setminus\D_2} K_{2,\sigma_k}(x,y)d(\xi-\xi_N)(s,x)d(\xi-\xi_N)(s,y)ds \\
&\ph + \frac{1}{2}\sum_{k=1}^\infty \int_0^t \frac{\vep^2}{\jp{\Fr_N^{avg}(\ux_N(s),\xi(s))}_{\vep}^3} \paren*{\int_{(\R^2)^2\setminus\D_2} K_{1,\sigma_k}(x,y)d(\xi-\xi_N)(s,x)d(\xi-\xi_N)(s,y)}^2 ds
\end{split}
\end{equation}
for every $0\leq t\leq T$. Here, $\jp{\cdot}_\vep\coloneqq (\vep^2 + (\cdot)^2)^{1/2}$, $u$ is the velocity field associated to $\xi$ through the Biot-Savart law, and $\xi_N=\frac{1}{N}\sum_{i=1}^N\d_{x_i}$ is the empirical measure associated to $\ux_N$.
\end{lemma}

Taking expectations of both sides of identity \eqref{eq:MR_ME_deriv} and using Fubini-Tonelli, we find that for any $\vep>0$,
\begin{equation}
\begin{split}
&\mathbf{E}\paren*{\jp{\Fr_N^{avg}(\ux_N(t),\xi(t))}_{\vep} - \jp{\Fr_N^{avg}(\ux_N(0),\xi(0))}_{\vep}} = \mathrm{Term}_1+\cdots+\mathrm{Term}_4,
\end{split}
\end{equation}
where
\begin{align}
\mathrm{Term}_1 &= \int_0^t \E\paren*{\frac{\Fr_N^{avg}(\ux_N(s),\xi(s))}{\jp{\Fr_N^{avg}(\ux_N(s),\xi(s))}_{\vep}}\int_{(\R^2)^2\setminus\D_2} K_{1,u}(x,y)d(\xi-\xi_N)(s,x)d(\xi-\xi_N)(s,y)}ds, \\
\mathrm{Term}_2 &= \frac{1}{2}\sum_{k=1}^\infty\int_0^t\E\paren*{\frac{\Fr_N^{avg}(\ux_N(s),\xi(s))}{\jp{\Fr_N^{avg}(\ux_N(s),\xi(s))}_{\vep}}\int_{(\R^2)^2\setminus\D_2} K_{1,(\sigma_k\cdot\nabla)\sigma_k}(x,y)d(\xi-\xi_N)(s,x)d(\xi-\xi_N)(s,y)}ds, \\
\mathrm{Term}_3 &= \frac{1}{2}\sum_{k=1}^\infty \int_0^t \E\paren*{\frac{\Fr_N^{avg}(\ux_N(s),\xi(s))}{\jp{\Fr_N^{avg}(\ux_N(s),\xi(s))}_{\vep}} \int_{(\R^2)^2\setminus\D_2} K_{2,\sigma_k}(x,y)d(\xi-\xi_N)(s,x)d(\xi-\xi_N)(s,y)}ds, \\
\mathrm{Term}_4 &= \frac{1}{2}\sum_{k=1}^\infty \int_0^t \E\paren*{\frac{\vep^2}{\jp{\Fr_N^{avg}(\ux_N(s),\xi(s))}_{\vep}^3} \paren*{\int_{(\R^2)^2\setminus\D_2} K_{1,\sigma_k}(x,y)d(\xi-\xi_N)(s,x)d(\xi-\xi_N)(s,y)}^2 }ds.
\end{align}
We now go to work on each of the $\mathrm{Term}_j$.

\begin{description}[leftmargin=*]
\item[Estimate for $\mathrm{Term}_1$]
Observe that $u$ is a.s. log-Lipschitz and by \cref{lem:PE_bnds} and conservation of the $L^\infty$ norm, we have the a.s. point-wise (in $\om\in\Omega$) bounds
\begin{align}
\|\xi(\om)\|_{L^\infty([0,T];L^\infty(\R^2))} &\lesssim \|\xi^0\|_{L^\infty(\R^2)},\\
\|u(\om)\|_{L^\infty([0,T]; LL(\R^2))} &\lesssim \|\xi^0\|_{L^\infty(\R^2)}.
\end{align}
So, we may apply \cref{prop:FOLLkprop} with $p=\infty$ point-wise in $(s,\om)$ to obtain that
\begin{equation}
\begin{split}
&\left|\int_{(\R^2)^2\setminus\D_2} K_{1,u}(x,y)d(\xi-\xi_N)(x)d(\xi-\xi_N)(y) \right| \\
&\lesssim |\ln\ep_3|\|\xi^0\|_{L^\infty(\R^2)}\paren*{\Fr_N^{avg}(\ux_N,\xi) + \frac{|\ln \ep_3|}{N} + \|\xi^0\|_{L^\infty(\R^2)}\ep_3^{2} + \ep_2|\ln\ep_2| (\frac{1}{\ep_3} + \|\xi^0\|_{L^\infty(\R^2)}^{1/2})},
\end{split}
\end{equation}
where $\ep_2, \ep_3: [0,\infty) \rightarrow (0,1)$ are measurable functions such that $1 \gg \ep_3(s)>\ep_2(s)>0$. Since $|r/\jp{r}_{\vep}|\leq 1$, we conclude from linearity of expectation that
\begin{equation}
\label{eq:MR_T1_fin}
\begin{split}
|\mathrm{Term}_1| &\lesssim \|\xi^0\|_{L^\infty(\R^2)}\int_0^t |\ln \ep_3(s)| \E(|\Fr_N^{avg}(\ux_N(s),\xi(s))|) ds \\
&\ph + \|\xi^0\|_{L^\infty(\R^2)}\int_0^t |\ln\ep_3(s)|\paren*{\frac{|\ln \ep_3(s)|}{N} + \|\xi^0\|_{L^\infty(\R^2)}\ep_3(s)^2 + \ep_2(s)|\ln\ep_2(s)|(\frac{1}{\ep_3(s)}+\|\xi^0\|_{L^\infty(\R^2)}^{1/2})}ds.
\end{split}
\end{equation}

\item[Estimate for $\mathrm{Term}_2$]
For each $k\in\N$, $(\sigma_k\cdot\nabla)\sigma_k \in W^{1,\infty}(\R^2;\R^2)$ by assumption. Since the vector field is divergence-free, an application of the product rule shows that it satisfies the gradient bound
\begin{equation}
\|\nabla ((\sigma_k\cdot\nabla)\sigma_k)\|_{L^\infty(\R^2)} \leq \|\nabla\sigma_k\|_{L^\infty(\R^2)}^2.
\end{equation}
Applying \cref{prop:FOLkprop} point-wise in $(s,\om)$ with $v=\sigma_k$ and then summing over $k$, we find that
\begin{align}
&\sum_{k=1}^\infty \left|\int_{(\R^2)^2\setminus\D_2} K_{1,(\sigma_k\cdot\nabla)\sigma_k}(x,y)d(\xi-\xi_N)(x)d(\xi-\xi_N)(y)\right| \nn\\
&\lesssim \sum_{k=1}^\infty \|\nabla\sigma_k\|_{L^\infty(\R^2)}^2 \paren*{\Fr_N^{avg}(\ux_N,\xi) + \frac{|\ln \ep_3|}{N} + \|\xi^0\|_{L^\infty(\R^2)}\ep_3^2 + \ep_1(\|\xi^0\|_{L^\infty(\R^2)}^{1/2} + \frac{1}{\ep_3})}.
\end{align}
where $\ep_1,\ep_3: [0,\infty)\rightarrow (0,1)$ are measurable functions such that $1\gg \ep_3(s)>2\ep_1(s)>0$. Since $|r/\jp{r}_{\vep}|\leq 1$, we conclude from the linearity of expectation that
\begin{equation}
\label{eq:MR_T2_fin}
\begin{split}
|\mathrm{Term}_2| &\lesssim \|\nabla\usig\|_{\ell_k^2L_x^\infty(\N\times\R^2)}^2\int_0^t \E(|\Fr_N^{avg}(\ux_N(s),\xi(s))|)ds \\
&\ph + \|\nabla\usig\|_{\ell_k^2L_x^\infty(\N\times\R^2)}^2\int_0^t \paren*{\frac{|\ln\ep_3(s)|}{N} + \|\xi^0\|_{L^\infty(\R^2)}\ep_3(s)^2 + \ep_1(s)(\|\xi^0\|_{L^\infty(\R^2)}^{1/2} +\frac{1}{\ep_3(s)})}ds.
\end{split}
\end{equation}

\item[Estimate for $\mathrm{Term}_3$]
For each $k\in\N$, $\sigma_k\in W^{1,\infty}(\R^2)$ by assumption. So applying \cref{prop:SOkprop} point-wise in $(s,\om)$ with $v=\sigma_k$, we find that
\begin{equation}
\begin{split}
&\sum_{k=1}^\infty \left|\int_{(\R^2)^2\setminus\D_2} K_{2,\sigma_k}(x,y)d(\xi-\xi_N)(x)d(\xi-\xi_N)(y) \right| \\
&\lesssim \|\nabla\usig\|_{\ell_k^2L_x^\infty(\N\times\R^2)}^2\paren*{ \Fr_N^{avg}(\ux_N,\xi) + \frac{|\ln\ep_3|}{N} + \|\xi^0\|_{L^\infty(\R^2)}\ep_3^2 + \ep_1(\|\xi^0\|_{L^\infty(\R^2)}^{1/2} + \frac{1}{\ep_3})},
\end{split}
\end{equation}
where $\ep_1,\ep_3$ are as above. By the same reasoning used to obtain \eqref{eq:MR_T2_fin}, it now follows that
\begin{equation}
\label{eq:MR_T3_fin}
\begin{split}
|\mathrm{Term}_3| &\lesssim \|\nabla\usig\|_{\ell_k^2 L_x^\infty(\N\times\R^2)}^2 \int_0^t \E(|\Fr_N^{avg}(\ux_N(s),\xi(s))|)ds \\
&\ph + \|\nabla\usig\|_{\ell_k^2L_x^\infty(\N\times\R^2)}^2\int_0^t \paren*{\frac{|\ln\ep_3(s)|}{N} + \|\xi^0\|_{L^\infty(\R^2)}\ep_3(s)^2 + \ep_1(s)(\|\xi^0\|_{L^\infty(\R^2)}^{1/2} +\frac{1}{\ep_3(s)})}ds.
\end{split}
\end{equation}

\item[Estimate for $\mathrm{Term}_4$]
Finally, we apply \cref{prop:FOLkprop} point-wise in $(s,\om)$ with $v=\sigma_k$ to obtain
\begin{equation}
\begin{split}
&\sum_{k=1}^\infty \left|\int_{(\R^2)^2\setminus\D_2} K_{1,\sigma_k}(x,y)d(\xi-\xi_N)(x)d(\xi-\xi_N)(y)\right|^2 \\
&\lesssim \|\nabla\usig\|_{\ell_k^2 L_x^\infty(\N\times\R^2)}^2\paren*{\Fr_N^{avg}(\ux_N,\xi)+\frac{|\ln \ep_3|}{N} + \|\xi^0\|_{L^\infty(\R^2)}\ep_3^2 + \ep_1(\|\xi^0\|_{L^\infty(\R^2)}^{1/2} + \frac{1}{\ep_3})}^2.
\end{split}
\end{equation}
Since we have the elementary inequality
\begin{equation}
\frac{\vep^2}{\jp{r}_{\vep}^3} \leq \frac{1}{\jp{r}_{\vep}} \leq \frac{1}{\vep} \qquad \forall r\in\R,
\end{equation}
it follows from the convexity of $z\mapsto z^2$ that if choose $\vep=(\ln N)/N$, then
\begin{equation}
\begin{split}
&\sum_{k=1}^\infty\E\paren*{\frac{\vep^2}{\jp{\Fr_N^{avg}(\ux_N(s),\xi(s))}_{\vep}^3} \paren*{\int_{(\R^2)^2\setminus\D_2} K_{1,\sigma_k}(x,y)d(\xi-\xi_N)(s,x)d(\xi-\xi_N)(s,y)}^2 } \\
&\lesssim  \|\nabla\usig\|_{\ell_k^2 L_x^\infty(\N\times\R^2)}^2 \E\paren*{\jp{\Fr_N^{avg}(\ux_N(s),\xi(s))}_{\frac{\ln N}{N}} }\\
&\ph +  \|\nabla\usig\|_{\ell_k^2 L_x^\infty(\N\times\R^2)}^2\frac{N}{\ln N}\paren*{\frac{|\ln\ep_3(s)|}{N} + \|\xi^0\|_{L^\infty(\R^2)}\ep_3(s)^2 +\ep_1(s)(\|\xi^0\|_{L^\infty(\R^2)}^{1/2} + \frac{1}{\ep_3(s)})}^2.
\end{split}
\end{equation}
After a little bookkeeping, we find that
\begin{equation}
\label{eq:MR_T4_fin}
\begin{split}
|\mathrm{Term}_4| &\lesssim \|\nabla\usig\|_{\ell_k^2 L_x^\infty(\N\times\R^2)}^2\int_0^t \E\paren*{\jp{\Fr_N^{avg}(\ux_N(s),\xi(s))}_{\frac{\ln N}{N}} }ds \\
&\ph + \frac{N\|\nabla\usig\|_{\ell_k^2 L_x^\infty(\N\times\R^2)}^2}{\ln N}\int_0^t \paren*{\frac{|\ln\ep_3(s)|}{N} + \|\xi^0\|_{L^\infty(\R^2)}\ep_3(s)^2 +\ep_1(s)(\|\xi^0\|_{L^\infty(\R^2)}^{1/2} + \frac{1}{\ep_3(s)})}^2 ds.
\end{split}
\end{equation}
\end{description}
\medskip

Combining our estimates \eqref{eq:MR_T1_fin}, \eqref{eq:MR_T2_fin}, \eqref{eq:MR_T3_fin}, and \eqref{eq:MR_T4_fin} for $\mathrm{Term}_1$, $\mathrm{Term}_2$, $\mathrm{Term}_3$, and $\mathrm{Term}_4$, respectively, we find that there exists a constant $C_1>0$ such that
\begin{equation}
\label{eq:MR_ep_pre}
\begin{split}
&\E\paren*{\jp{\Fr_N^{avg}(\ux_N(t),\xi(t))}_{\frac{\ln N}{N}}} - \E\paren*{\jp{\Fr_N^{avg}(\ux_N(0),\xi(0))}_{\frac{\ln N}{N}}} \\
&\leq C_1\paren*{\|\xi^0\|_{L^\infty(\R^2)} + \|\nabla\usig\|_{\ell_k^2 L_x^\infty(\N\times\R^2)}^2}\int_0^t |\ln \ep_3(s)| \E\paren*{\jp{\Fr_N^{avg}(\ux_N(s),\xi(s))}_{\frac{\ln N}{N}}}ds \\
&\ph + C_1\|\xi^0\|_{L^\infty(\R^2)}\int_0^t |\ln \ep_3(s)|\paren*{\frac{|\ln\ep_3(s)|}{N} + \|\xi^0\|_{L^\infty(\R^2)}\ep_3(s)^2 + \ep_2(s)|\ln\ep_2(s)|(\frac{1}{\ep_3(s)}+\|\xi^0\|_{L^\infty(\R^2)}^{1/2})}ds \\
&\ph + C_1\|\nabla\usig\|_{\ell_k^2L_x^\infty(\N\times\R^2)}^2\int_0^t \paren*{\frac{|\ln\ep_3(s)|}{N} + \|\xi^0\|_{L^\infty(\R^2)}\ep_3(s)^2 + \ep_1(s)(\|\xi^0\|_{L^\infty(\R^2)}^{1/2} + \frac{1}{\ep_3(s)})}ds \\
&\ph + \frac{C_1\|\nabla\usig\|_{\ell_k^2 L_x^\infty(\N\times\R^2)}^2 N}{\ln N}\int_0^t \paren*{\frac{|\ln\ep_3(s)|}{N} + \|\xi^0\|_{L^\infty(\R^2)}\ep_3(s)^2 + \ep_1(s)(\|\xi^0\|_{L^\infty(\R^2)}^{1/2} + \frac{1}{\ep_3(s)})}^2 ds .
\end{split}
\end{equation}
We choose the time-dependent functions $\ep_1,\ep_2$ according to
\begin{align}
\ep_1(s) &= \ep_2(s)^2,\\
\ep_2(s)|\ln \ep_2(s)| &= \ep_3(s)^2.
\end{align}
We now introduce the maximal function
\begin{equation}
\label{eq:GN_def}
\G_N(t) \coloneqq \sup_{0\leq s\leq t} \E\paren*{\jp{\Fr_N^{avg}(\ux_N(s),\xi(s))}_{\frac{\ln N}{N}}}.
\end{equation}
With this notation and substituting these choices into the right-hand side of \eqref{eq:MR_ep_pre} and simplifying, we obtain the inequality
\begin{equation}
\label{eq:MR_ep3_pre}
\begin{split}
\G_N(t) - \G_N(0) &\leq C_2\paren*{\|\xi^0\|_{L^\infty(\R^2)}+\|\nabla\usig\|_{\ell_k^2 L_x^\infty(\N\times\R^2)}^2}\int_0^t |\ln\ep_3(s)| \G_N(s)ds \\
&\ph + C_2\|\xi^0\|_{L^\infty(\R^2)}\int_0^t |\ln\ep_3(s)|\paren*{\frac{|\ln\ep_3(s)|}{N} + \ep_3(s)(1+\|\xi^0\|_{L^\infty(\R^2)}\ep_3(s))}ds \\
&\ph + C_2\|\nabla\usig\|_{\ell_k^2 L_x^\infty(\N\times\R^2)}^2\int_0^t \paren*{\frac{|\ln\ep_3(s)|}{N}+(\ep_3(s)+\|\xi^0\|_{L^\infty(\R^2)})\ep_3(s)^2}ds \\
&\ph + \frac{C_2\|\nabla\usig\|_{\ell_k^2 L_x^\infty(\N\times\R^2)}^2 N}{\ln N}\int_0^t \paren*{\frac{|\ln\ep_3(s)|}{N} + (\ep_3(s)+\|\xi^0\|_{L^\infty(\R^2)})\ep_3(s)^2}^2 ds,
\end{split}
\end{equation}
where $C_2\geq C_1$ is a possibly larger constant. For each $s\in[0,T]$, we choose
\begin{equation}
\ep_3(s) \coloneqq \min\{\G_N(s), e^{-1}, \|\xi^0\|_{L^\infty(\R^2)}^{-1}\},
\end{equation}
which is evidently a measurable function. Note that $(\ln N)/N\leq \ep_3(s)$, provided that $N\gg 1$. In fact, one can check from the continuity of the map $s\mapsto \E(\jp{\Fr_N^{avg}(\ux_N(s),\xi(s))}_{\ln N/N})$ that $\G_N$ is also continuous (cf. the proof of \cite[Lemma 5.4]{Rosenzweig2019_LL}). Substituting this choice for $\ep_3(s)$ into the right-hand side of inequality \eqref{eq:MR_ep3_pre} and performing a bit of algebra, we find that
\begin{equation}
\begin{split}
\label{eq:MR_pre_Os}
\G_N(t) &\leq \G_N(0) + C_3\paren*{\|\xi^0\|_{L^\infty(\R^2)}+\|\nabla\usig\|_{\ell_k^2 L_x^\infty(\N\times\R^2)}^2}\int_0^t |\ln\G_N(s)| \G_N(s)ds \\
&\ph + C_3\paren*{\|\xi^0\|_{L^\infty(\R^2)} + \|\nabla\usig\|_{\ell_k^2 L_x^\infty(\N\times\R^2)}^2}\frac{(\ln(N/\ln N))^2 t}{N},
\end{split}
\end{equation}
where $C_3\geq C_2$ is a possibly larger constant.

To close the estimate \eqref{eq:MR_pre_Os} using the Osgood lemma (recall \cref{lem:Os}), we argue as follows. Fix a time $t\in (0,\infty)$. Let $N\in\N$ be sufficiently large so that
\begin{equation}
\label{eq:MR_N_con}
C_3\paren*{\|\xi^0\|_{L^\infty(\R^2)} + \|\nabla\usig\|_{\ell_k^2 L_x^\infty(\N\times\R^2)}^2}t < \ln\ln\paren*{\G_N(0) + \frac{C_3 t (\|\xi^0\|_{L^\infty(\R^2)}+\|\nabla\usig\|_{\ell_k^2 L_x^\infty(\N\times\R^2)}^2)(\ln\frac{N}{\ln N})^2}{N}}^{-1}.
\end{equation}
By continuity of the function $\G_N$, there exists a time $0<t_N^*\leq t$ (with the convention that $t_N^* = t$ if no such time exists) such that
\begin{equation}
\G_N(s) < e^{-1}, \quad \forall 0\leq s\leq t_N^* \qquad \text{and} \qquad \G_N(t_N^*)=e^{-1}.
\end{equation}
Applying \cref{lem:Os} with modulus of continuity $r\mapsto r\ln(1/r)$ for $r\in [0,e^{-1}]$ together with \cref{rem:Os_ex_log}, it follows from the condition \eqref{eq:MR_N_con} that for every $0\leq s\leq t_N^*$, 
\begin{align}
\G_N(s) &\leq \exp\Bigg(-\exp\Bigg(\ln\ln\Big(\G_N(0) + \frac{C_3s(\|\xi^0\|_{L^\infty(\R^2)} + \|\nabla\usig\|_{\ell_k^2L_x^\infty(\N\times\R^2)}^2)(\ln\frac{N}{\ln N})^2}{N}\Big)^{-1}  \nn\\
&\hspace{50mm} - C_3s\paren*{\|\xi^0\|_{L^\infty(\R^2)} + \|\nabla\usig\|_{\ell_k^2 L_x^\infty(\N\times\R^2)}^2}\Bigg)\Bigg) \nn\\
&=\paren*{\G_N(0) + \frac{C_3 s (\|\xi^0\|_{L^\infty(\R^2)}+\|\nabla\usig\|_{\ell_k^2 L_x^\infty(\N\times\R^2)}^2)(\ln\frac{N}{\ln N})^2}{N}}^{e^{-C_3s(\|\xi^0\|_{L^\infty(\R^2)} + \|\nabla\usig\|_{\ell_k^2 L_x^\infty(\N\times\R^2)}^2)}}
\end{align}
and that the expression in the ultimate line is $<e^{-1}$. Thus, $t_N^*=t$ and therefore the proof of \cref{thm:main} is complete.

\appendix
\section{Singular integral operators}
\label{sec:app}
In the appendix, we review single integral operators (SIOs) not of Calder\'{o}n-Zygmund type, the so-called \emph{Calder\'{o}n $d$-commutators} from the work \cite{CJ1987} of Christ and Journ\'{e}. This review ultimately leads up to our proof that the matrix-valued SIO defined in \cref{ssec:intro_RM} of the introduction has an $L^2$-bounded extension.

\subsection{Multilinear singular integral forms}
We start with the basics of singular integral forms, closely following the presentation of Christ and Journ\'{e}.

\begin{mydef}[$\d$-BSIF]
\label{def:BSIF}
For $\d>0$, a \emph{$\d$-bilinear singular integral form ($\d$-BSIF)} is a mapping $T:(C_c^\infty(\R^d))^2\rightarrow\C$ such that if $f,g\in C_c^\infty(\R^d)$ have disjoint supports, then
\begin{equation}
\label{eq:BSIF_ker}
T(g,f) = \int_{(\R^d)^2} K(x,y)g(x)f(y)dxdy,
\end{equation}
where the kernel $K: (\R^d)^2\setminus\D_2 \rightarrow\C$ satisfies
\begin{align}
&|K(x,y)| \lesssim \frac{1}{|x-y|^d}, \label{eq:BSIF_size}\\
&|K(x,y)-K(x',y)| \lesssim \frac{|x-x'|^\d}{|x-y|^{d+\d}} \qquad\qquad \forall |x-x'|\leq \frac{|x-y|}{2}, \label{eq:BSIF_smooth_x}\\
&|K(y,x)-K(y,x')| \lesssim \frac{|x-x'|^\d}{|x-y|^{d+\d}} \qquad\qquad \forall |x-x'| \leq \frac{|x-y|}{2}. \label{eq:BSIF_smooth_y}
\end{align}
The best implicit constant in \eqref{eq:BSIF_size} is denoted by $|K|_0$ and in \eqref{eq:BSIF_smooth_x} and \eqref{eq:BSIF_smooth_y} by either $|K|_\d$ or $|T|_\d$.
\end{mydef}

As the reader may check, we can define a $\d$-BSIF on the domain $C_{c0}^\infty(\R^d)\times C^\infty(\R^d)$ or $C^\infty(\R^d)\times C_{c0}^\infty(\R^d)$, where $C_{c0}^\infty(\R^d)\subset C_c^\infty(\R^d)$ is the subspace consisting of mean-zero functions. Therefore, we can define the elements $T_1(1), T_2(1)\in (C_{c0}^\infty(\R^d))'$ by
\begin{align}
\ipp{g,T_1(1)} &= T(g,1) \qquad \forall g\in C_{c0}^\infty(\R^d), \label{eq:T_1(1)}\\
\ipp{T_2(1),f} &= T(1,f) \qquad \forall f\in C_{c0}^\infty(\R^d). \label{eq:T_2(1)}
\end{align}

\begin{mydef}[WBP]
\label{def:WBP}
A $\d$-BSIF $T$ has the \emph{weak boundedness property (WBP)} if for every pair $(f,g) \in (C_{c}^\infty(\R^d))^2$ satisfying
\begin{equation}
\max\{\diam(\supp f), \diam(\supp g)\} \leq 4t,
\end{equation}
it holds that
\begin{equation}
\label{eq:WBP}
|T(g,f)| \lesssim t^d\paren*{\|g\|_{L^\infty(\R^d)} + t\|\nabla g\|_{L^\infty(\R^d)}}\paren*{\|f\|_{L^\infty(\R^d)} + t\|\nabla f\|_{L^\infty(\R^d)}}.
\end{equation}
The best implicit constant in \eqref{eq:WBP} is denoted by $|T|_{W}$.
\end{mydef}

\begin{mydef}[Bounded $\d$-BSIF]
\label{def:BSIF_bnd}
A $\d$-BSIF $T$ is said to be \emph{bounded} if
\begin{equation}
\label{eq:BSIF_bnd}
|T(g,f)| \lesssim \|f\|_{L^2(\R^d)} \|g\|_{L^2(\R^d)} \qquad \forall f,g\in C_c^\infty(\R^d).
\end{equation}
We denote the best constant implicit in \eqref{eq:BSIF_bnd} by $\|T\|_{2,2}$, and we define the quantity $\|T\|_{\d} \coloneqq |K|_{\d}+\|T\|_{2,2}$, where $|K|_{\d}$ is as in \cref{def:BSIF}.
\end{mydef}

The classical $T(1)$ theorem of David and Journ\'{e} shows that the boundedness of $T$ is equivalent to the distributions $T_1(1), T_2(1)$ defined in \eqref{eq:T_1(1)}, \eqref{eq:T_2(1)} above belonging to the space BMO of bounded mean oscillation \emph{and} $T$ having the WBP of \cref{def:WBP}.
\begin{thm}[$T(1)$ theorem \cite{DJ1984, CM1986}]
\label{thm:T(1)}
The $\d$-BSIF $T$ is bounded if and only if $T_1(1),T_2(1)\in BMO(\R^d)$ and $T$ has the WBP. Moreover,
\begin{equation}
\|T\|_{2,2} \lesssim \paren*{\|T_1(1)\|_{BMO(\R^d)} + \|T_2(1)\|_{BMO(\R^d)} + |T|_{W}} + c_\d|T|_{\d}.
\end{equation}
\end{thm}

We now introduce the multilinear generalization of $\d$-BSIFS.
\begin{mydef}[$\d$-n SIF]
For $0<\d\leq 1$ and integer $n\geq 2$, a \emph{$\d$-$n$-linear singular integral form ($\d$-$n$ SIF)} is a mapping $U:(C_c^\infty(\R^d))^n\rightarrow\C$ with the following property. For every $1\leq i<j\leq N$ and $h_{m_1},\ldots,h_{m_{n-2}}\in C_c^\infty(\R^d)$, where $m_k\in\{1,\ldots,n\}\setminus \{i,j\}$ and $m_1<\cdots<m_{n-2}$, define the bilinear form
\begin{equation}
\begin{split}
U_{ij}(h_{m_1},\ldots,h_{m_{n-2}}): (C_c^\infty(\R^d))^2\rightarrow \C \\
U_{ij}(h_{m_1},\ldots,h_{m_{n-2}})(h_i,h_j) \coloneqq  U(h_1,\ldots,h_n).
\end{split}
\end{equation}
Then $U_{ij}(h_{m_1},\ldots,h_{m_{n-2}})$ is a $\d$-BSIF and
\begin{equation}
\label{eq:Uij_BSIF}
U_{ij}(h_{m_1},\ldots,h_{m_{n-2}}) \lesssim_{i,j} \prod_{{1\leq k\leq n}\atop{k\notin\{i,j\}}} \|h_k\|_{L^\infty(\R^d)}.
\end{equation}
We denote the best constant implicit in \eqref{eq:Uij_BSIF} by $|U_{ij}|_\d$ and define $|U|_{\d}\coloneqq \sup_{1\leq i<j\leq n} |U_{ij}|_{\d}$. For every $1\leq i\leq n$ and $1\leq i<j\leq n$, we denote the best implicit constants in the estimates\footnote{One can show that these estimates are, in fact, equivalent (see \cite[Theorem A]{CJ1987}).}
\begin{align}
|U(f_1,\ldots,f_n)| &\lesssim_i \paren*{\prod_{{1\leq k\leq n}\atop{k\neq i}} \|f_k\|_{L^\infty(\R^d)}} \|f_i\|_{\mathscr{H}^1(\R^d)}, \qquad \forall f_1,\ldots,f_n\in C_c^\infty(\R^d) \\
|U(f_1,\ldots,f_n)| &\lesssim_{i,j} \paren*{\prod_{{1\leq k\leq n}\atop{k\notin\{i,j\}}}\|f_k\|_{L^\infty(\R^d)}} \|f_i\|_{L^2(\R^d)} \|f_j\|_{L^2(\R^d)}, \qquad \forall f_1,\ldots,f_n\in C_c^\infty(\R^d)
\end{align}
respectively by $\|U\|_{i}$ and $\|U\|_{ij}$, where $\mathscr{H}^1(\R^d)$ denotes the Hardy space. We say that $U$ is \emph{bounded} if
\begin{equation}
\|U\| \coloneqq \max_{1\leq i<j\leq n}\{\|U\|_{i}, \|U\|_{ij}\} < \infty.
\end{equation}
\end{mydef}

To each $\d$-$n$ SIF $U$ and integer $1\leq m\leq n$, we can define a mutlilinear operator
\begin{equation}
\begin{split}
&\pi_{U}^{(m)}:(C_c^\infty(\R^d))^{n-1} \rightarrow (C_c^\infty(\R^d))',\\
&\ipp{h_m, \pi_U^{(m)}(h_1,\ldots,h_{m-1},h_{m+1},\ldots,h_n)} \coloneqq U(h_1,\ldots,h_n).
\end{split}
\end{equation}
As in the bilinear case, $U(f_1,\ldots,f_n)$ remains well-defined when one $f_i\in C_{c0}^\infty(\R^d)$ and all the other $f_j\in C^\infty(\R^d)$. For each $1\leq i\leq n$, we can then define $U_i(1)\in (C_{c0}^\infty(\R^d))'$ by
\begin{equation}
\ipp{g,U_i(1)} \coloneqq U(\underbrace{1,\ldots,1}_{i-1},g,\underbrace{1,\ldots,1}_{n-i}), \qquad \forall g\in C_{c0}^\infty(\R^d).
\end{equation}
To generalize the bilinear WBP, \cref{def:WBP}, to the multilinear case, we introduce the Fourier multiplier
\begin{equation}
\wh{P_t f}(\xi) \coloneqq \wh{\varphi}(t\xi)\wh{f}(\xi),
\end{equation}
where $\varphi \in C_c^\infty(\R^d)$ is a nonnegative, radial function with unit mean.
\begin{mydef}[$\d$-$n$ WBP]
\label{def:WBP_n}
We say that the $\d$-$n$ SIF has the \emph{WBP} if for every pair $1\leq i<j\leq n$, all $t>0$ and $f_i,f_j\in C_c^\infty(\R^d)$ satisfying
\begin{equation}
\max_{i,j}\{\diam(\supp f_i), \diam(\supp f_j)\} \leq 4t,
\end{equation}
it holds for all $f_k\in C_c^\infty(\R^d)$, $k\notin\{i,j\}$, that
\begin{equation}
\label{eq:WBP_n_ic}
\begin{split}
&|U(P_t f_1,\ldots, P_t f_{i-1}, f_i, P_t f_{i+1},\ldots,P_t f_{j-1},f_j, P_t f_{j+1},\ldots, P_t f_n)|\\
&\lesssim_{i,j} \paren*{\prod_{{1\leq k\leq n}\atop {k\notin\{i,j\}}} \|f_k\|_{L^\infty(\R^d)}}t^d\paren*{\|f_i\|_{L^\infty(\R^d)}+t\|\nabla f_i\|_{L^\infty(\R^d)}}\paren*{\|f_j\|_{L^\infty(\R^d)}+t\|\nabla f_j\|_{L^\infty(\R^d)}}.
\end{split}
\end{equation}
We denote the best implicit constant in \eqref{eq:WBP_n_ic} by $|U_{ij}|_{w}$ and define $|U|_{w} \coloneqq \max_{1\leq i<j\leq n} |U_{ij}|_{w}$.
\end{mydef}

\begin{remark}
The constants in \cref{def:WBP_n} implicitly depend on the function $\varphi$ underlying the definition of the operator $P_t$; however, this dependence will not be important, as $\varphi$ is fixed. Additionally, the definition of $|U_{ij}|_{w}$ is not quite the same as in the bilinear WBP \cref{def:WBP} due to the use of $P_t$.
\end{remark}

The following theorem due to Christ and Journ\'{e} is the multilinear generalization of \cref{thm:T(1)}.
\begin{thm}[{\cite[Theorem 2]{CJ1987}}]
\label{thm:T(1)_n}
A $\d$-$n$ SIF $U$ is bounded if and only if it has the WBP \emph{and} $U_i(1)\in BMO(\R^d)$ for every $1\leq i\leq n$. Moreover,
\begin{equation}
\|U\| \lesssim_{\d} \sum_{i=1}^n \|U_i(1)\|_{BMO(\R^d)} +n^2(|U|_{\d}+|U|_{w}).
\end{equation}
\end{thm}

\subsection{Calder\'{o}n $d$-commutators}
\label{ssec:app_CdC}
We now recall the class of Calder\'{o}n $d$-commutators, a (nontrival) higher-dimensional generalization of the classical Calder\'{o}n commutators. Let $T$ be a Calder\'{o}n-Zygmund convolution operator bounded on $L^2(\R^d)$. Let $K(x,y) = K(x-y)$ satisfying conditions \eqref{eq:BSIF_size}, \eqref{eq:BSIF_smooth_x}, and \eqref{eq:BSIF_smooth_y} denote the convolution kernel associated to $T$ in the sense of \eqref{eq:BSIF_ker}. For $a\in C^\infty(\R^d)$, we define
\begin{equation}
\label{eq:m_xy}
m_{x,y}a \coloneqq \int_0^1 a(tx+(1-t)y)dt \qquad \forall x\neq y\in\R^d.
\end{equation}
Then for $f_1,\ldots,f_{n+2}\in C_c^\infty(\R^d)$, the integral
\begin{equation}
\label{eq:CdC}
\int_{(\R^d)^2} K(x-y)\paren*{\prod_{i=1}^n m_{x,y}f_i}f_{n+1}(x)f_{n+2}(y)dxdy
\end{equation}
is well-defined if $\supp f_{n+1}, \supp f_{n+2}$ are disjoint and determines an $(n+2)$-linear form denoted by $W$. Note that if the kernel $K$ has a cancellation property, such as zero average over annuli, then \eqref{eq:CdC} is well-defined in the principal value sense without restriction on the supports of $f_{n+1}, f_{n+2}$. The main result from \cite{CJ1987} that we need is the following theorem establishing the boundedness of $W$.

\begin{thm}[{\cite[Theorem 3]{CJ1987}}]
\label{thm:CJ}
For every $\d>0$ and $n\in\N$,
\begin{equation}
|W(f_1,\ldots,f_{n+2})| \lesssim_{\d} n^{2+\d}\paren*{\prod_{i=1}^n \|f_i\|_{L^\infty(\R^d)}} \|f_{n+1}\|_{L^2(\R^d)} \|f_{n+2}\|_{L^2(\R^d)}, \qquad \forall f_1,\ldots,f_{n+2}\in C_c^\infty(\R^d).
\end{equation}
\end{thm}

\begin{remark}
As the reader can check, assuming $K$ has zero average over annuli, \cref{thm:CJ} implies by a density argument that the multilinear form $W$ has a well-defined extension on $(L^\infty(\R^d))^n\times (L^2(\R^d))^2$.
\end{remark}

We close this subsection by noting that highly nontrivial improvements to and generalizations of the Christ-Journ\'{e} result \cref{thm:CJ} have been given in the subsequent years by Seeger, Smart, and Street \cite{SSS2019} and Lai \cite{Lai2018}. Since we do not need such refinements for the purposes of this article, we have limited our attention to the original work \cite{CJ1987} of Christ and Journ\'{e}.

\subsection{The operators $T_{v}$ and $T_{\usig}$}
\label{ssec:T_sig}
We encountered in the introduction the operator $T_{\usig}$ defined by
\begin{equation}
\label{eq:app_Tsig_def}
T_{\usig} f(x) \coloneqq \int_{\R^2} K_{\usig}(x,y)f(y)dy, \qquad K_{\usig}(x,y) \coloneqq \sum_{k=1}^\infty \nabla^2\g(x-y) : (\sigma_k(x)-\sigma_k(y))^{\otimes 2}.
\end{equation}
The goal of this subsection is to show that such operators are smoothing of order two, in the sense that
\begin{equation}
\|\nabla T_{\usig}(\nabla f)\|_{L^2(\R^2;(\R^2)^{\otimes 2})} \lesssim_{\usig} \|f\|_{L^2(\R^2)} \qquad \forall f\in C_c^\infty(\R^2),
\end{equation}
and therefore $\nabla T_{\usig}\nabla$ extends to a bounded operator $L^2(\R^2)\rightarrow L^2(\R^2;(\R^2)^{\otimes 2})$. As part of our analysis, we also will show that operators of the form
\begin{equation}
T_{v}f(x) \coloneqq \int_{\R^2} K_{v}(x,y)f(y)dy, \qquad K_v(x,y)\coloneqq \nabla\g(x-y) \cdot \paren*{v(x)-v(y)},
\end{equation}
which, as we saw in the introduction to the article, appear in the It\^o equation \eqref{eq:ME_deriv} satisfied by the modulated energy $\Fr_N^{avg}(\ux_N(t),\xi(t))$, are also smoothing by two orders. To emphasize the second-order and first-order nature of the kernels defining $T_{\usig}$ and $T_v$, respectively, from hereafter we write $K_{2,\usig}, T_{2,\usig}$ and $K_{1,v}, T_{1,v}$.

\subsubsection{Warm-up: smoothing of $T_{1,v}$}
To warm up, we show that $T_{1,v}$ has the desired smoothing by two orders property using \cref{thm:CJ}.
\begin{prop}
\label{prop:T_1v}
Let $v\in C_c^\infty(\R^2)$. Then we have that
\begin{equation}
\|\nabla T_{1,v}(\nabla f)\|_{L^2(\R^2; (\R^2)^{\otimes 2})} \lesssim \|\nabla v\|_{L^\infty(\R^2)} \|f\|_{L^2(\R^2)} \qquad \forall f\in C_c^\infty(\R^2).
\end{equation}
Consequently, for any $\al,\beta \in \{1,2\}$, the form
\begin{equation}
(C_c^\infty(\R^2))^3\rightarrow\C, \qquad (v,f,g) \mapsto \ipp{g, \p_{\al} T_{1,v}(\p_{\beta} f)}
\end{equation}
has a bounded extension to $\text{Lip}(\R^2)\times (L^2(\R^2))^2$.
\end{prop}
\begin{proof}
We first compute the Schwartz kernel of $\nabla T_{1,v}\nabla$ as a continuous linear mapping $\Sc(\R^2)\rightarrow \Sc'(\R^2;(\R^2)^{\otimes 2})$. Fixing two test functions $f,g$ and indices $\al,\be\in\{1,2\}$, we have by duality that
\begin{align}
\ipp{f,\p_\al T_{1,v}(\p_\be g)} &= -\ipp{\p_{\al}f, T_{1,v}(\p_{\be}g)} =-\lim_{\d\rightarrow 0^+} \int_{|x-y|\geq\d} K_{1,v}(x,y) \p_{\al}f(x)\p_{\be}g(y)dxdy,
\end{align}
where the ultimate equality follows from unpacking the definition of $T_{1,v}$ and applying dominated convergence. Integrating by parts once in both $x$ and $y$, we find that the right-hand side equals
\begin{equation}
\begin{split}
&\lim_{\d\rightarrow 0^+}\int_{|x-y|=\d} \frac{(x-y)_\al}{|x-y|} K_{1,v}(x,y) f(x)\p_\be g(y)d\H^3(x,y) \\
&+\lim_{\d\rightarrow 0^+} \int_{|x-y|= \d}\frac{(x-y)_\be}{|x-y|} \p_{x_\al}K_{1,v}(x,y)f(x)g(y)d\H^3(x,y) \\
&-\lim_{\d\rightarrow 0^+}\int_{|x-y|\geq\d} \p_{x_\al}\p_{y_\be}K_{1,v}(x,y)f(x)g(y)dxdy,
\end{split}
\end{equation}
where $\H^d$ denotes the $d$-dimensional Hausdorff measure. By direct estimation, we see that the first term vanishes. For the second term, we fist observe that
\begin{equation}
\p_{x_\al} K_{1,v}(x,y) = -\frac{1}{2\pi}\paren*{\frac{\d_{\al\ga}}{|x-y|^2} - \frac{2(x-y)_\al(x-y)_\ga}{|x-y|^4}}\paren*{v(x)-v(y)}^\ga - \frac{(x-y)_\ga\p_\al v^\ga(x)}{2\pi |x-y|^2},
\end{equation}
where we have implicitly used the convention of Einstein summation. So by dilation and translation invariance and Fubini-Tonelli,
\begin{equation}
\begin{split}
&-\int_{|x-y|=\d}\frac{(x-y)_\be}{|x-y|}\p_{x_\al}K_{1,v}(x,y)f(x)g(y)d\H^3(x,y)\\
 &= \frac{1}{2\pi}\int_{\R^2}dy g(y) \int_{\p B(0,1)} z_\be z\cdot\p_\al v(y+\d z) f(y+\d z) d\H^1(z) \\
&\ph +\frac{1}{2\pi}\int_{\R^2}dy g(y)\int_{\p B(0,1)}z_\be(\d_{\al\ga}-2z_\al z_\ga)z^\rho m_{y+\d z,y}(\p_\rho v^{\ga}) f(y+\d z)d\H^1(z).
\end{split}
\end{equation}
As $\d\rightarrow 0^+$, dominated convergence implies that the preceding right-hand side converges to
\begin{equation}
C_{\be\ga}\ipp{g,\p_{\al}v^\ga f} + C_{\be\ga}^{\al\rho}\ipp{g,\p_{\rho}v^\ga f},
\end{equation}
where $C_{\be\ga}, C_{\be\ga}^{\al\rho}$ are constants defined by
\begin{equation}
C_{\be\ga} \coloneqq -\frac{1}{2\pi}\int_{\p B(0,1)}z_\be z_\ga d\H^1(z), \qquad C_{\be\ga}^{\al\rho} \coloneqq -\frac{1}{2\pi}\int_{\p B(0,1)} z_{\be}(\d_{\al\ga}-2z_{\al}z_{\ga})z^{\rho}d\H^1(z).
\end{equation}
Thus, after a little a bookkeeping, we conclude that
\begin{equation}
\label{eq:T_1v_PV}
\ipp{f,\p_\al T_{1,v}(\p_\be g)} = C_{\be\ga}\ipp{f,\p_{\al}v^\ga g} + C_{\be\ga}^{\al\rho}\ipp{f,\p_{\rho}v^{\ga} g}-\PV\int_{(\R^2)^2} \p_{x_\al}\p_{y_\be}K_{1,v}(x,y)f(x)g(y)dxdy.
\end{equation}

Since multiplication by $\p_{\al}v^{\ga},\p_{\rho}v^{\ga}$ is $L^2$-bounded by H\"{o}lder's inequality, we only need to show that the principal value term in identity \eqref{eq:T_1v_PV} defines a bounded trilinear form. We now check that $-\p_{x_\al}\p_{y_\be}K_{1,v}$ may be put in the form satisfying the conditions of \cref{thm:CJ}. To this end, observe the identity
\begin{equation}
\begin{split}
&2\pi\p_{x_\al}\p_{y_\be}K_{1,v}(x,y)\\
&= \paren*{-\frac{2(\d_{\al\ga}(x-y)_{\be} + \d_{\al\be}(x-y)_\ga + \d_{\ga\be}(x-y)_\al)}{|x-y|^4} + \frac{8(x-y)_\al(x-y)_\ga(x-y)_\be}{|x-y|^6}}\paren*{v(x)-v(y)}^{\ga} \\
&\ph + \paren*{\frac{\d_{\al\ga}}{|x-y|^2} -\frac{2(x-y)_\al(x-y)_\ga}{|x-y|^4}} \p_\be v^\ga(y) + \paren*{\frac{\d_{\ga\be}}{|x-y|^2} - \frac{2(x-y)_\ga(x-y)_\be}{|x-y|^4}}\p_{\al}v^\ga(x),
\end{split}
\end{equation}
which is valid for $x\neq y\in\R^2$ and follows by direct computation. Note that by the fundamental theorem of calculus, the expression in the second line may be written as
\begin{equation}
\label{eq:T_1v_CJ}
K_{1,v, \al\be\ga}^{(2),\rho}(x-y) m_{x,y}(\p_{\rho}v^\ga),
\end{equation}
where the reader will recall the definition of $m_{x,y}$ from \eqref{eq:m_xy} and
\begin{equation}
K_{1,v,\al\be\ga}^{(2),\rho}(z) \coloneqq -2z_\rho\paren*{\frac{\d_{\al\ga}z_\be + \d_{\al\be}z_\ga + \d_{\ga\be}z_\al}{|z|^4} - \frac{4z_\al z_\be z_\ga}{|z|^6}} \qquad \forall z\neq 0.
\end{equation}
Similarly, the expression in the third line may be written as
\begin{equation}
K_{1,v,\al\be\ga}^{(1)}(x-y)\p_{\be}v^\ga(y) + K_{1,v,\be\al\ga}^{(1)}(x-y)\p_{\al}v^\ga(x),
\end{equation}
where
\begin{equation}
K_{1,v,\al'\be'\ga'}^{(1)}(z) \coloneqq \frac{\d_{\al'\ga'}}{|z|^2} -\frac{2z_{\al'}z_{\ga'}}{|z|^4}, \qquad \forall z\neq 0.
\end{equation}
$K_{1,v,\al'\be'\ga'}^{(1)}$ and $K_{1,v,\al'\be'\ga'}^{(2),\rho'}$ are even, homogeneous of degree $-2$ kernels, which are smooth on the sphere $S^1$. Moreover, it is a tedious, but not hard, exercise to show using trigonometric identities that they have zero average on $S^1$. So by \cite[Theorem 5.2.10]{grafakos2014c}, they define $L^2$-bounded Calder\'{o}n-Zygmund operators of convolution type. Therefore, the singular integral form defined by \eqref{eq:T_1v_CJ} is bounded on $L^2(\R^2)$ by \cref{thm:CJ}. Since multiplication by $\p_{\be}v^\ga$ or $\p_{\al}v^\ga$ is also $L^2$-bounded by H\"older's inequality, the proof of the proposition is complete after a little bookkeeping.
\end{proof}

\subsubsection{Smoothing of $T_{2,\usig}$}
We now proceed to showing that the operator $T_{2,\usig}$ has the desired order-two smoothing property. We begin by computing the Schwartz kernel of $\nabla T_{2,v}\nabla$, for any fixed smooth vector field $v$, as a continuous linear mapping $\Sc(\R^2)\rightarrow\Sc'(\R^2;(\R^2)^{\otimes 2})$. Before proceeding to this computation, we record useful identities for the first- and second-order partial derivatives of the kernel $K_{2,v}$.

\begin{lemma}
\label{lem:Ksig_PDs}
For $\al',\beta'\in\{1,2\}$ and $x\neq y\in\R^2$, we have the point-wise identities
\begin{equation}
\label{eq:Ksig_1PD}
\begin{split}
2\pi\p_{x_{\al'}}K_{2,v}(x,y) &= \paren*{\frac{2(\d_{\al\beta}(x-y)_{\al'} + \d_{\al\al'}(x-y)_{\beta}+\d_{\al'\beta}(x-y)_{\al})}{|x-y|^4} - \frac{8(x-y)_{\al}(x-y)_{\beta}(x-y)_{\al'}}{|x-y|^6}} \\
&\hspace{25mm} (v(x)-v(y))^\al(v(x)-v(y))^\beta \\
&\ph + \paren*{-\frac{\d_{\al\beta}}{|x-y|^2}+\frac{2(x-y)_\al(x-y)_{\beta}}{|x-y|^4}}\paren*{\p_{\al'}v^\al(x)(v(x)-v(y))^{\beta} + (v(x)-v(y))^{\al}\p_{\al'}v^\beta(x)}
\end{split}
\end{equation}
and
\begin{equation}
\label{eq:Ksig_2PD}
\begin{split}
-2\pi\p_{x_{\al'}}\p_{y_{\beta'}}K_{\sigma}(x,y) &= \Bigg(\frac{2(\d_{\al\be}\d_{\al'\be'}+\d_{\al\al'}\d_{\be\be'}+\d_{\al'\be}\d_{\al\be'})}{|x-y|^4}-\frac{8\d_{\al\be}(x-y)_{\al'}(x-y)_{\be'}}{|x-y|^6}  \\
&\hspace{15mm} -\frac{8(\d_{\al\al'}(x-y)_{\be}+\d_{\al'\be}(x-y)_\al)(x-y)_{\be'}}{|x-y|^6} \\
&\hspace{15mm} -\frac{8(\d_{\al\be'}(x-y)_{\be}(x-y)_{\al'}+\d_{\be\be'}(x-y)_{\al}(x-y)_{\al'}+\d_{\al'\be'}(x-y)_{\al}(x-y)_{\be})}{|x-y|^6} \\
&\hspace{15mm} + \frac{48(x-y)_{\al}(x-y)_{\be}(x-y)_{\al'}(x-y)_{\be'}}{|x-y|^8}\Bigg) (v(x)-v(y))^\al (v(x)-v(y))^\be \\
&\ph +  \Bigg(\frac{2(\d_{\al\be}(x-y)_{\al'} + \d_{\al\al'}(x-y)_\be + \d_{\al'\be}(x-y)_\al)}{|x-y|^4} - \frac{8(x-y)_\al(x-y)_\be(x-y)_{\al'}}{|x-y|^6}\Bigg) \\
&\hspace{15mm} \paren*{\p_{\be'}v^\al(y)(v(x)-v(y))^\be + \p_{\be'}v^\be(y)(v(x)-v(y))^\al} \\
&\ph + \Bigg(\frac{2(\d_{\al\be}(x-y)_{\be'} + \d_{\al\be'}(x-y)_\be + \d_{\be'\be}(x-y)_\al)}{|x-y|^4} - \frac{8(x-y)_\al(x-y)_\be(x-y)_{\be'}}{|x-y|^6}\Bigg) \\
&\hspace{15mm} \paren*{\p_{\al'}v^\al(x)(v(x)-v(y))^\be + \p_{\al'}v^\be(x)(v(x)-v(y))^\al} \\
&\ph + \paren*{-\frac{\d_{\al\be}}{|x-y|^2}+\frac{2(x-y)_{\al}(x-y)_{\be}}{|x-y|^4}} \paren*{\p_{\al'}v^\al(x)\p_{\be'}v^\be(y) + \p_{\be'}v^\al(y)\p_{\al'}v^\be(x) }.
\end{split}
\end{equation}
\end{lemma}

\begin{lemma}
\label{lem:Tsig_SK}
For any $f,g\in \Sc(\R^2)$ and $\al',\beta'\in \{1,2\}$, we have that
\begin{equation}
\begin{split}
\ipp{f,\p_{\al'}T_{2,v}(\p_{\beta'}g)} &= \ipp{f,C_{\al\beta\al'}^{\gamma\gamma'\be'} \p_{\gamma}v^\al\p_{\gamma'}v^\beta g} + \ipp{f, C_{\al\be}^{\ga\be'}(\p_{\al'}v^{\al}\p_{\ga}v^{\be} + \p_{\al'}v^{\be}\p_{\ga}v^{\al})g}\\
&\ph -\PV\int_{(\R^2)^2}\p_{x_{\al'}}\p_{y_{\be'}}K_{2,v}(x,y)f(x)g(y)dxdy,
\end{split}
\end{equation}
where $C_{\al\beta\al'}^{\gamma\gamma'\be'}, C_{\al\be}^{\ga\be'}$ are real constants and we use the convention of Einstein summation.
\end{lemma}
\begin{proof}
Let $f,g\in \Sc(\R^2)$. Proceeding by duality, we have that for any $\al',\beta'\in\{1,2\}$,
\begin{align}
\ipp{f,\p_{\al'}T_{2,v}(\p_{\beta'}g)} &= -\ipp{\p_{\al'}f, T_{2,v}(\p_{\beta'}g)} \nn\\
&=-\lim_{\d\rightarrow0^+} \int_{|x-y|\geq \d} K_{2,v}(x,y) \p_{\al'}f(x)\p_{\be'}g(y)dxdy,
\end{align}
where the ultimate equality is by dominated convergence. Integrating by parts first in $y$ and then in $x$, we find that the preceding expression equals
\begin{align}
&\lim_{\d\rightarrow 0^+} \int_{|x-y|=\d} \frac{(x-y)_{\al'}}{|x-y|} K_{2,v}(x,y)f(x)\p_{\beta'}g(y)d\H^3(x,y) \label{eq:SK_I}\\
&+\lim_{\d\rightarrow 0^+}\int_{|x-y|=\d}\frac{(x-y)_{\be'}}{|x-y|}\p_{x_{\al'}}K_{2,v}(x,y)f(x)g(y)d\H^3(x,y) \label{eq:SK_II} \\
&-\lim_{\d\rightarrow 0^+} \int_{|x-y|\geq \d}\p_{x_{\al'}}\p_{y_{\beta'}}K_{2,v}(x,y)f(x)g(y)dxdy, \label{eq:SK_III}
\end{align}
where the reader will recall that $\H^d$ denotes the $d$-dimensional Hausdorff measure.

It is straightforward to check that \eqref{eq:SK_I} equals zero. Applying \cref{lem:Ksig_PDs}, a change of variable, and the Fubini-Tonelli theorem, we have that
\begin{equation}
\begin{split}
\eqref{eq:SK_II} &= \frac{1}{2\pi}\int_{\R^2}dy g(y)\int_{S^1}d\H^1(z)f(y+\d z)z_{\be'}\paren*{2(\d_{\al\beta}z_{\al'}+\d_{\al\al'}z_{\beta} + \d_{\al'\beta}z_{\al}) -8z_{\al}z_{\beta}z_{\al'}}z^{\ga}z^{\ga'}\\
&\ph\hspace{45mm}  m_{y+\d z,y}(\p_\ga v^\al)m_{y+\d z,y}(\p_{\ga'}v^{\be}) \\
&\ph + \frac{1}{2\pi}\int_{\R^2}dy g(y)\int_{S^1}d\H^1(z)f(y+\d z)z_{\be'}(-\d_{\al\be}+2z_{\al}z_{\be})z^{\ga} \\
&\ph\hspace{45mm}\paren*{(\p_{\al'}v^{\al})(y+\d z)m_{y+\d z,y}(\p_{\ga}v^{\be}) + (\p_{\al'}v^{\be})(y+\d z)m_{y+\d z,y}(\p_{\ga}v^{\al})},
\end{split}
\end{equation}
where we also use the fundamental theorem of calculus applied to the $v$. By the dominated convergence theorem, this last expression tends, as $\d\rightarrow 0^+$, to
\begin{equation}
C_{\al\beta\al'}^{\gamma\gamma'\be'}\int_{\R^2}f(y)g(y)\p_{\gamma}v^\al(y)\p_{\gamma'}v^{\beta}(y)dy + C_{\al\be}^{\ga\be'}\int_{\R^2}f(y)g(y)\paren*{\p_{\al'}v^{\al}(y)\p_{\ga}v^{\be}(y) + \p_{\al'}v^{\be}(y)\p_{\ga}v^{\al}(y)}dy,
\end{equation}
where $C_{\al\beta\al'}^{\gamma\gamma'\be'}, C_{\al\be}^{\ga\be'}$ are the constants defined by
\begin{align}
C_{\al\beta\al'}^{\gamma\gamma'\be'} &\coloneqq \frac{1}{2\pi}\int_{S^1}\paren*{2(\d_{\al\beta}z_{\al'}+\d_{\al\al'}z_{\beta} + \d_{\al'\beta}z_{\al}) -8z_{\al}z_{\beta}z_{\al'}}z_{\gamma}z_{\gamma'}z_{\be'}d\H^1(z), \\
C_{\al\be}^{\ga\be'} &\coloneqq \frac{1}{2\pi}\int_{S^1}(-\d_{\al\be}+2z_{\al}z_{\be})z_{\be'}z_{\ga}d\H^1(z).
\end{align}
Thus,
\begin{equation}
\eqref{eq:SK_II} = \ipp{f, C_{\al\beta\al'}^{\gamma\gamma'\be'}\p_\gamma v^\al\p_{\gamma'}v^\beta g} + \ipp{f, C_{\al\be}^{\ga\be'}(\p_{\al'}v^{\al}\p_{\ga}v^{\be} + \p_{\al'}v^{\be}\p_{\ga}v^{\al})g}.
\end{equation}

After a little bookkeeping, we conclude that
\begin{equation}
\begin{split}
\ipp{f,\p_{\al'}T_{2,v}(\p_{\beta'}g)} &= \ipp{f,C_{\al\beta\al'}^{\gamma\gamma'\be'} \p_{\gamma}v^\al\p_{\gamma'}v^\beta g} + \ipp{f, C_{\al\be}^{\ga\be'}(\p_{\al'}v^{\al}\p_{\ga}v^{\be} + \p_{\al'}v^{\be}\p_{\ga}v^{\al})g}\\
&\ph -\lim_{\d\rightarrow 0^+}\int_{|x-y|\geq \d}\p_{x_{\al'}}\p_{y_{\beta'}}K_{2,v}(x,y)f(x)g(y)dxdy,
\end{split}
\end{equation}
which completes the proof of the lemma.
\end{proof}

By H\"older's inequality, multiplication by $\p_{\ga}v^\al \p_{\ga'}v^\be$ is bounded on $L^2(\R^2)$. So from \cref{lem:Tsig_SK}, we see that in order to prove $\nabla T_{2,v}\nabla$ is $L^2$-bounded, we need to show that the principal value operator
\begin{equation}
\label{eq:PV_op}
-\PV \int_{(\R^2)^2} \p_{x_{\al'}}\p_{y_{\be'}}K_{2,v}(x,y)g(y)dy
\end{equation}
defines a bounded operator on $L^2(\R^2)$, for any indices $\al',\be'\in\{1,2\}$. To do this, we want to use \cref{thm:CJ}.

Using the identity from \cref{lem:Ksig_PDs} for $-\p_{x_{\al'}}\p_{y_{\be'}}K_{2,v}(x,y)$ and the fundamental theorem of calculus, we may write for $x\neq y$ that 
\begin{equation}
\begin{split}
-\p_{x_{\al'}}\p_{y_{\be'}}K_{2,v}(x,y) &= K_{2,v,\al\be\al'\be'}^{(2),\gamma\gamma'}(x-y) m_{x,y}(\p_\gamma v^\al) m_{x,y}(\p_{\gamma'}v^\be) \\
&\ph +  K_{2,v,\al\be\al'}^{(1),\ga'}(x-y) m_{x,y}(\p_{\ga'}v^\be)\p_{\be'}v^\al(y) + K_{2,v,\al\be\al'}^{(1),\ga}(x-y) m_{x,y}(\p_\ga v^\al)\p_{\be'}v^\be(y) \\
&\ph +  K_{2,v,\al\be\be'}^{(1), \ga'}(x-y)\p_{\al'}v^\al(x)m_{x,y}(\p_{\ga'}v^\be) + K_{2,v,\al\be\be'}^{(1), \ga}(x-y)\p_{\al'} v^\be(x)m_{x,y}(\p_\ga v^\al) \\
&\ph +  K_{2,v,\al\be}^{(0)}(x-y)\paren*{\p_{\al'}v^\al(x)\p_{\be'}v^\be(y)+\p_{\be'}v^\al(y)\p_{\al'}v^\be(x)}
\end{split}
\end{equation}
where we use Einstein summation and for $z\neq 0$,
\begin{equation}
\begin{split}
K_{2, v, \al\be\al'\be'}^{(2),\ga\ga'}(z) &= \frac{1}{2\pi} \Bigg(\frac{2(\d_{\al\be}\d_{\al'\be'}+\d_{\al\al'}\d_{\be\be'}+\d_{\al'\be}\d_{\al\be'})}{|z|^4}-\frac{8(\d_{\al\be}z_{\al'}z_{\be'}+\d_{\al\al'}z_{\be}z_{\be'}+\d_{\al'\be}z_\al z_{\be'})}{|z|^6}   \\
&\hspace{20mm} -\frac{8(\d_{\al\be'}z_{\be}z_{\al'}+\d_{\be\be'}z_{\al}z_{\al'}+\d_{\al'\be'}z_{\al}z_{\be})}{|z|^6} + \frac{48z_{\al}z_{\be}z_{\al'}z_{\be'}}{|z|^8}\Bigg)z_\ga z_{\ga'},
\end{split}
\end{equation}
\begin{equation}
\begin{split}
K_{2,v,\al\be\rho}^{(1),\nu}(z) &= \frac{1}{2\pi}\Bigg(\frac{2(\d_{\al\be}z_{\rho}+\d_{\al\rho}z_\be + \d_{\rho\be}z_\al)}{|z|^4}  - \frac{8z_\al z_\be z_{\rho}}{|z|^6}\Bigg) z_{\nu},
\end{split}
\end{equation}
and
\begin{equation}
\begin{split}
K_{2,v,\al\be}^{(0)}(z) &= \frac{1}{2\pi}\Bigg(-\frac{\d_{\al\be}}{|z|^2}+\frac{2z_\al z_\be}{|z|^4}\Bigg).
\end{split}
\end{equation}
Each of the kernels $K_{2,v,\al\be\al'\be'}^{(2), \ga\ga'}, K_{2,v,\al\be\rho}^{(1),\nu}, K_{2,v,\al\be}^{(0)}$ previously defined is even, homogeneous of degree $-2$ and smooth on $S^1$. Moreover, each has zero average on $S^1$. It then follows from \cite[Theorem 5.2.10]{grafakos2014c} that the associated Calder\'{o}n-Zygmund convolution operators are bounded on $L^2(\R^2)$.

We now have all the necessary ingredients to use \cref{thm:CJ} to show that $-\p_{x_\al}\p_{y_\be}K_\sigma(x,y)$ is the kernel of an $L^2$ bounded principal value operator. The following proposition now follows, as the reader may check, after a little bookkeeping of our preceding work.

\begin{prop}
\label{prop:T_2v}
Let $v\in C_c^\infty(\R^2)$. Then we have that
\begin{equation}
\|\nabla T_{2,v}(\nabla f)\|_{L^2(\R^2;(\R^2)^{\otimes 2})} \lesssim \|\nabla v\|_{L^\infty(\R^2)}^2 \|f\|_{L^2(\R^2)} \qquad \forall f\in C_c^\infty(\R^2).
\end{equation}
Consequently, for any $\al,\be\in\{1,2\}$, the form
\begin{equation}
(C_c^\infty(\R^2))^3 \rightarrow \C, \qquad (v,f,g) \mapsto \ipp{g,\p_\al T_{2,v}(\p_\be f)}
\end{equation}
has a bounded extension to $\text{Lip}(\R^2)\times (L^2(\R^2))^2$.
\end{prop}

\bibliographystyle{siam}
\bibliography{library}
\end{document}